\documentclass[11pt,reqno]{amsart}
\numberwithin{equation}{section}
\usepackage{amssymb}
\usepackage{enumerate, xspace}
\usepackage{marvosym} 
\usepackage[sans]{dsfont} 
\usepackage{bbm}
\usepackage{changebar}
\usepackage[colorlinks]{hyperref}
\usepackage{verbatim}
\usepackage[usenames,dvipsnames,svgnames,table]{xcolor}
\usepackage{graphicx}
\usepackage{ulem}
\usepackage{setspace}
\usepackage[charter]{mathdesign}
\newtheorem{thm}{Theorem}[section]
\newtheorem{lem}[thm]{Lemma}
\newtheorem{cor}[thm]{Corollary}

\newtheorem{prop}[thm]{Proposition}

\newtheorem{definition}[thm]{Definition}

\newtheorem{rem}[thm]{Remark}

\newcommand\zr{\nu_{ss}^N}
\newcommand\ex{\mu_{ss}^N}
\newcommand\Lzr{L_N}
\newcommand\Lex{\mathcal{L}_N}
\newcommand\xalpha{\tilde\alpha}
\newcommand\xbeta{\tilde \beta}
\newcommand\zalpha{\alpha}
\newcommand\zbeta{\beta}

\newcommand\ve{\varepsilon}

\newcommand\vf{\varphi}

\newcommand\LL{{\mathbb L}}
\newcommand\EE{{\mathbb E}}

\newcommand\NN{{\mathbb N}}
\newcommand\RR{{\mathbb R}}

\newcommand\ZZ{{\mathbb Z}}

\newcommand\R{{\mathbb R}}

\usepackage{comment}

\newcommand{\mc}[1]{{\mathcal #1}}

\newcommand{\bb}[1]{{\mathbb #1}}

\begin{document}
\title[Boundary driven zero-range with long jumps]{Non-equilibrium stationary properties of the\\ boundary driven zero-range process with long jumps}
\author[Bernardin]{\textsc{C\'edric Bernardin}} 
\address{Universit\'e C\^ote d'Azur, CNRS, LJAD\\
Parc Valrose\\
06108 NICE Cedex 02, France\\
\& Interdisciplinary Scientific Center Poncelet (CNRS IRL 2615), 119002 Moscow, Russia}
\email{{\tt cbernard@unice.fr}}
\author[Gon\c{c}alves]{\textsc{Patr\'icia   Gon\c calves}}
\address{Center for Mathematical Analysis,  Geometry and Dynamical Systems \\Instituto Superior T\'ecnico, Universidade de Lisboa\\
Av. Rovisco Pais, no. 1, 1049-001 Lisboa, Portugal}
\email{{pgoncalves@tecnico.ulisboa.pt}}
\author[Oviedo]{\textsc{Byron Jim\'enez-Oviedo}}
\address{Escuela de Matem\'atica,  \\
Faculdad de Ciencias Exactas y Naturales, Universidad Nacional de Costa Rica\\
Heredia, Costa Rica}
\email{{byron.jimenez.oviedo@una.cr}}
\author[Scotta]{\textsc{Stefano Scotta}}
\address{Center for Mathematical Analysis,  Geometry and Dynamical Systems \\
Instituto Superior T\'ecnico, Universidade de Lisboa\\
Av. Rovisco Pais, no. 1, 1049-001 Lisboa, Portugal}
\email{{stefano.scotta@tecnico.ulisboa.pt}}

\date{\today.}
\begin{abstract} 
We consider the zero-range process with long jumps and in contact with {infinitely extended reservoirs} in its non-equilibrium stationary state. We derive the hydrostatic limit and the Fick's law, which are a consequence of a {static} relationship between the  exclusion process and the zero-range process. We also obtain the large deviation principle for the empirical density, i.e. we compute the non-equilibrium free energy.
\end{abstract}


\maketitle


\section{Introduction}

The description of the macroscopic properties of the non-equilibrium stationary state (NESS) of a large system of interacting particles driven outside of equilibrium by boundary forces has seen a lot of activity and progress recently.   The results on the  NESS follow from the combination of two different approaches: one based on ad hoc exact computations of the NESS \cite{DER} and the other one, based on  the so-called Macroscopic Fluctuation Theory   (MFT) \cite{BER-MFT}, which provides a unified beautiful treatment of non-equilibrium systems described macroscopically by diffusive or hyperbolic conservation laws.
MFT is more generic, but computationally less efficient, and it is based on the development of the hydrodynamic limit theory \cite{Spohnbook, KL}. MFT is usually applied to diffusive systems, i.e. systems whose hydrodynamic equation is given by  a diffusion equation, but it is, in fact, more general and the framework encapsulates also the systems whose hydrodynamic limits are described by conservation laws \cite{J, VAR, Bah, Mar0, BBMN, Bar2}. It is also possible to use the MFT for the NESS of systems with several microscopic conservation laws, see e.g. \cite{Ber0, BD, Ber1,V0}.

\medskip 

More recently, several studies of interacting particle systems have appeared, whose hydrodynamic limits are given by a fractional diffusion equation or a fractional conservation law \cite{J0, BPS, cpb, BGBJ, BGS, BJ, CGJ, G0, GJ, GS0, Seth, SethS}. In order to derive such equations, the authors have to invoke a suitable coarse-graining in space and time which is not the diffusive one - space does not scale as square root of time -- nor the Eulerian one -- space does not scale as time. Since we will only consider NESS in this paper, i.e. the system in infinite time, we prefer to not to enter in these space-time scales considerations here. In the models considered in those articles, the fractional nature is induced by the presence of non-local interactions in the microscopic dynamics (the reader can think for example of a system of independent random walks with a transition probability which has infinite variance). Since the interactions are non-local, the consequences of  the boundary reservoirs is more subtle to understand, with respect to the case of local interactions, because the operators involved (microscopically and macroscopically) are non-local while a boundary condition has a local nature. One of the motivations to study such systems with non-local interactions is that they could play the role of toy models to describe some ``fractional'' universality classes \cite{LP, Spohn0, DDSMS, PSS,PSSS,HG, Spohn_Stolz, Kundu} of interacting particles with local interactions but with several conservation laws (see \cite{JKO,BGJ, BGJS0, BGJSS, BGJS, SSS,C0} for rigorous studies).

\medskip 

The aim of this article is to provide a rigorous study of the NESS of a superdiffusive interacting particle system whose macroscopic behaviour is described by a fractional diffusion. Apart from \cite{BJ} we are not aware of any rigorous study of the NESS for superdiffusive models. Our study focuses on the boundary driven zero-range process with long jumps. The zero-range process with finite range jumps has been introduced in \cite{Spi} and then intensively studied with different aims: existence theorems for the infinite volume dynamics, characterization of the invariant measures, derivation of hydrodynamic limits, study of phase transitions and condensation phenomena etc, \cite{Lig,DMF,Spohnbook,KL,evans0,evans,LMS,CG}. Our choice to investigate the zero-range process with long jumps instead of some other process is based on two reasons: first, it is one of the rare systems for which the NESS has some semi-explicit form; second, its NESS is strongly connected to the NESS of the boundary driven exclusion process with long jumps. The later is not explicit but it has the advantage to have "simple"  hydrodynamic limits, in the sense that they are given by linear equations (but non-local and with boundary conditions), so that the hydrostatic properties of the NESS of the boundary driven exclusion process with long jumps are available.  To be precise, we consider a one-dimensional superdiffusive zero-range process with long jumps and in contact with extended reservoirs at the boundary of the domain of the jumping particles. By connecting properties of its typical profile in the NESS with the ones of the boundary driven exclusion process with long jumps, and using the fact that its NESS is of product type, we derive the form of its hydrostatic profile,  we also prove a fractional Fick's law and moreover, we derive the large deviation function of the empirical density in the NESS. While the hydrodynamic limits are not derived in this work, here we  provide certainly a first step in the development of the MFT for superdiffusive systems.

\medskip 

This paper is organised as follows. In Section \ref{sec:model and Results} we describe the boundary driven zero range process with long jumps in contact with reservoirs (for simplicity ZRP) and the link between its NESS and the NESS of the boundary driven exclusion process with long jumps. We also present there the main results obtained in this paper. In Section \ref{sec:3} we present the proof of two results which give information about the NESS for the ZRP and are the building blocks for our main theorem.  In Section \ref{sec:4} we present the proof of our main results, which is a generalization of the hydrostatic limit for the ZRP and the Fractional Fick's law. In Section \ref{sec:5} we present the proof of the Large Deviations for our model. Appendix \ref{app:hydrostaticsexclusion} is dedicated to the presentation of hydrostatic limit and Fick's law for the exclusion process with long jumps and in contact with reservoirs.

\section{Models and statement of results}
\label{sec:model and Results}

\subsection{The models}

The boundary driven zero-range process with long jumps is a  continuous time pure jump Markov process with countable state space $\Omega_{N}= \NN_{0}^{\Lambda_{N}}$ where $\Lambda_N=\{1, \ldots, N-1\}$, $N \ge 2$. A typical configuration of this process $\xi \in \Omega_N$ is denoted by  $\{\xi(x)\}_{x\in \Lambda_{N}}$ and $\xi(x)$ represents the number of particles at site $x\in \Lambda_{N}$. Its dynamics is defined through a non-decreasing  function  $g:\NN_{0}\rightarrow [0,\infty)$ such that $g(0)=0$  and strictly positive on the set of positive integers, 
and a  transition probability $p:\ZZ\rightarrow [0,1]$ given by
\begin{equation}
    \label{eq:prob}
p(z)= \cfrac{c_{\gamma}}{|z|^{1+\gamma}}, \quad |z| \ge 1,\quad p(0)=0,\end{equation} 
where  $0<\gamma<2$ and   $c_\gamma = 2/  {\zeta (\gamma+1)}>0$ is a normalisation factor with $\zeta$ being the Riemann zeta function. The parameter $d$ which will appear later is defined, for $\gamma>1$, by 
\begin{equation}
\label{eq:d}
d:=d(\gamma) =\sum_{z\ge 1} z p(z) =\zeta (\gamma)
\end{equation}
and corresponds to the half of the first moment of $p(\cdot)$. We remark that since $0<\gamma<2$ the second moment of $p(\cdot)$ is infinite.  
Before {describing} the process under investigation in this article, we first describe the zero-range process with long jumps and free boundary.

\subsubsection{The zero-range process with long jumps and free boundary}

The zero-range process $({\tilde \xi}_t)_{t\ge 0}$ with long jumps in $\Lambda_{N}$, with interaction rate $g(\cdot)$, transition probability $p(\cdot)$ and free boundary conditions is the pure jump Markov process on $\Omega_N$ generated by the operator $\Lzr^{b}$ acting on any bounded measurable function  $f:\Omega_N \to \RR$ as 
$$(\Lzr^b f)(\xi) = \, \sum_{x,y \in \Lambda_N} p (y-x)g(\xi(x))[ f(\xi^{x,y}) -f(\xi)].$$
Here the configuration $\xi^{x,y}$ denotes the configuration $\xi$ (we can always assume that $\xi(x)\ge 1$ since $g(0)=0$) obtained by moving one particle from $x$ to $y$, i.e. it is defined as 
\begin{equation}\label{interch_1}
\xi^{x,y}(z) =( \xi(x)-1)\mathbbm{1}_{z=x}+(\xi(y)+1)\mathbbm{1}_{z=y}+ \xi(z)\mathbbm{1}_{z\neq x,y}.
\end{equation}
The dynamics $({\tilde \xi}_t)_{t\ge 0}$ just defined  preserves the  number of particles. In fact, restricted to the subspace of $\Omega_N$ composed of configurations $\xi$ with a given fixed  number  of particles,  the process is ergodic and it has a unique invariant measure. Therefore, it follows that on $\Omega_N$ the process has a one-parameter family of invariant measures which are defined as follows. Consider the  partition function $Z: \RR^{+} \rightarrow \RR^{+}$ 
defined by 
$$Z(\vf) = \sum _{k =0}^{\infty}\dfrac{\vf^{k}}{g(k)!},$$
where  $g(k)!=g(1)\cdots g(k)$ for $k>0$ and $g(0)!= 1 $. {By the ratio test it is not difficult to see that $\varphi^{*} :=\liminf_{k\to \infty} g(k) \in (0, \infty]$ is the radius of convergence of the entire function $Z$.}  Moreover, $Z$ is strictly increasing.

For any $\vf < \vf^{*}$ we define on $\Omega_{N}$ the product measure $\tilde \nu_{\vf} $, with marginal distributions given by  
\begin{equation}
\tilde\nu_{\vf}\lbrace \xi \in \Omega_{N}\  ;  \ \xi(x) = k \rbrace = \dfrac{\vf^{k}}{Z(\vf)g(k)!}. 
\end{equation}

A remarkable property of the zero-range process is that $\{{ \tilde \nu}_\varphi\; ; \; 0\le \varphi < \varphi^*\}$ forms a family of invariant measures for the process generated by $L_N^b$.
Note that  for $\vf\in [0,\vf*)$ and any $x \in \Lambda_N$ we have that 
\begin{equation*}
E_{\tilde\nu_{\vf}}[g(\xi(x))]=\vf .
\end{equation*}
Hereinafter $E_\mu$ (resp. $\mathbb E_\mu$) denotes the expectation with respect to the probability measure $\mu$ (resp. the expectation with respect to the path probability measure $\mathbb P_\mu$ corresponding to the process starting with initial measure $\mu$).

For every $\vf\in [0,\vf^*)$, we denote by $R(\varphi)$ the average number of particles per site under ${ \tilde\nu}_\varphi$:
\begin{equation*} \label{eq:function_R}
R(\vf)= E_{\tilde \nu_{\vf}}[\xi(x)].\end{equation*}
The  function $R(\cdot)$ can be rewritten as 
\begin{equation}
\label{eq:Rprime}
R(\vf) = \dfrac{1}{Z(\vf)}\sum _{k=0}^{\infty}\dfrac{k \vf^{k}}{g(k)!}= \dfrac{\vf Z'(\vf)}{Z(\vf)} = \vf \dfrac{\rm d}{{\rm d} \vf}(\log(Z(\vf))).
\end{equation}
Since $R'(\vf) = {E}_{\tilde\nu_{\vf}}\left[\left(\xi(x)-{E}_{\nu_{\vf}}[\xi(x)]\right)^2\right]>0$, the function $R(\cdot)$ is strictly increasing from $[0,\vf^*)$ to $[0,\infty)$, and defining
\begin{equation*}
\displaystyle m^{*}:=\lim_{\vf\uparrow\vf^{*}} R(\vf),
\end{equation*}
the map $\vf\to R(\vf)$ is a bijection between $[0,\vf^{*})$ and $\left[0,\displaystyle m^{*}\right)$. We then denote by $\Phi(\cdot)$ the inverse map of $R(\cdot)$. Hence, we can alternatively parameterise the invariant measures by $m\in [0,m^{*})$, the number the particles per site,  instead of $\varphi$, i.e. we define for any $m\in [0,m^{*})$ that
$ \nu_{m} = \tilde\nu_{\Phi(m)}.$

 

\subsubsection{The boundary driven zero-range process with long jumps}\label{ZRPdef}

In order to define the boundary driven zero-range process with long jumps, we have now to introduce the boundary driving process. We will use infinitely extended particle reservoirs injecting or removing particles everywhere in the bulk $\Lambda_N$, { so that the number of particles is no longer conserved} (see, for example, \cite{cpb, BGBJ,BJ,BPS}). Any configuration $\xi \in \Omega_N$ is extended into a configuration $\xi \in ({\mathbb N}_0 \cup \{\infty\})^{\mathbb Z}$ by setting $\xi(z)=\infty$ if $z\notin \Lambda_N$. We also adopt the usual conventions of summation on ${\mathbb N}_0 \cup\{ \infty\}$. We assume that $\zalpha, \zbeta \in (0,m^{*})$ and without loss of generality that $\alpha \le \beta$.  Observe that
\begin{equation}
\label{eq:alphabeta}
\zalpha, \zbeta \in (0,m^{*}) \quad \Longleftrightarrow \quad \Phi (\zalpha), \Phi(\zbeta) \in (0, \varphi^*).  \end{equation}

The action of the generators on any bounded measurable function  $f:\Omega_N \to \RR$ at the left and right boundary are defined, respectively, by 
\begin{equation}
\begin{split}
&(\Lzr^{r} f)(\xi)= \sum_{x \in \Lambda_N}r_N^+(\tfrac xN)\left\lbrace \Phi(\zbeta) [f(\xi^{x,+}) - f(\xi)]+g(\xi(x))[ f(\xi^{x,-}) -f(\xi)]\right\rbrace,\\
&(\Lzr^{\ell} f)(\xi) = \sum_{x \in \Lambda_N}r_N^-(\tfrac xN)\left\lbrace \Phi(\zalpha)  [f(\xi^{x,+}) - f(\xi)]+g(\xi(x))[ f(\xi^{x,-}) -f(\xi)]\right\rbrace.
\end{split}
\end{equation}
Above
\begin{equation}\label{interch_glauber_zr}
\xi^{x,\pm}(z) =( \xi(x)\pm 1)\mathbbm{1}_{z=x}+ \xi(z)\mathbbm{1}_{z\neq x}
\end{equation}
and
\begin{equation}
\label{eq:functions_r}
r_{N}^{+}\left(\tfrac{x}{N}\right)= \sum_{y\ge N}p(y-x) \quad\textrm{and}\quad r_{N}^{-}\left(\tfrac{x}{N}\right)= \sum_{y\le0 }p(y-x).\end{equation}
The  zero-range process in contact with infinitely extended reservoirs  is the continuous time pure jump Markov process $(\xi_{t})_{t\ge 0}$ with infinitesimal generator given by 
\begin{equation}
\label{eq:generatorZR}
\Lzr = \underbrace{\Lzr^b}_{\text{Bulk dynamics}}  + \frac{\kappa}{ N^{\theta}}\underbrace{ \Big(\Lzr^{r} + \Lzr^{\ell}\Big)}_{\text{Boundary dynamics}}
\end{equation}
where  the parameters $\theta$ and $\kappa$ satisfy $\theta \in \R$ and $\kappa>0$.  We consider the process speeded up in the time scale \begin{equation}\label{eq:time_scale}
    \Theta(N)=n^\gamma\bb{1}_{\theta\geq 0}+n^{\gamma+\theta}\bb{1}_{\theta < 0}
\end{equation} so that its generator becomes $\Theta(N)L_N$.

\begin{figure}[h]
    \centering
    \includegraphics{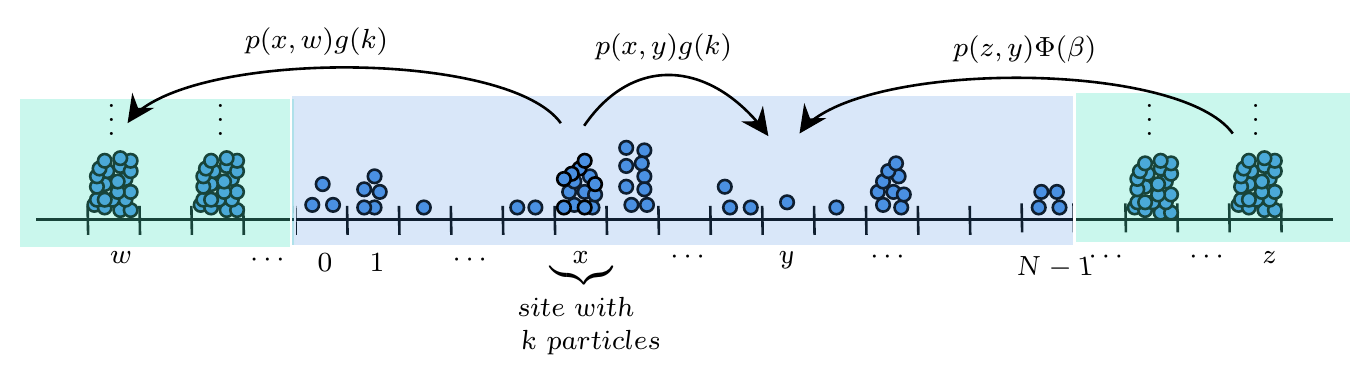}\caption{Dynamics of the Long-jumps Zero-range process in contact with infinitely extended reservoirs.}
\end{figure}

\begin{rem}
Other models of reservoirs could have been considered but the results would be quite similar (see Section 2.6 in \cite{BPS}). 
\end{rem}

\begin{rem}
Observe that we did not impose very restrictive conditions on the function $g$ defining the dynamics. Since we do not have to consider the dynamics in infinite volume but only in finite volume, existence of the dynamics can be obtained by rather standard methods. In the case of free boundaries, this is trivial because the dynamics is conservative in the number of particles so that if the initial condition has a finite number $M$ of particles then the dynamics will evolve on the finite state space composed of configurations with $M$ particles and thus it will be well defined. In the boundary driven case, this is less trivial because the number of particles is no longer conserved but we observe that there exists a constant $C$ (depending on $\alpha, \beta$ and $N$) such that 
$$L_N \left(\sum_{x\in \Lambda_N} \xi (x) \right) \le C.$$
This implies that if we start from a configuration with $M$ particles, then for any  time horizon  $T>0$ the dynamics will be well defined and it will evolve during the time interval $[0,T]$ on the finite state space composed of particles with at most $M(T)=M+CT$ particles.  
\end{rem}

The boundary driven exclusion process with long jumps has been introduced and studied in a series of recent works \cite{cpb, BGBJ, BJ,BPS}. It is not the model we are interested in this paper, nevertheless, it has some links with the boundary driven zero-range process that will be crucial to establish some of our results.

\subsubsection{The boundary driven exclusion process with long jumps}\label{excldef}

 The boundary driven $M$-exclusion process with long jumps is the continuous time  pure jump Markov process, that we denote by $(\eta_t)_{t\ge 0}$, with state space $\chi_N=\{0,1\}^{\Lambda_N}$ whose dynamics is defined as follows. A typical configuration $\eta\in \chi_N$  is denoted by  $\{\eta(x)\}_{x\in \Lambda_{N}}$ with  $\eta (x)\in \{0,1\}$, for $x \in \Lambda_N$, where we interpret $\eta (x) =1$ (resp. $\eta (x)=0$) as the presence (resp. the absence) of a particle at site $x$. Its infinitesimal generator  is given by 
\begin{equation}\label{genexcl}
\Lex = \underbrace{\Lex^b}_{\text{Bulk dynamics}}  + \frac{\kappa}{ N^{\theta}}\underbrace{ \Big(\Lex^{r} + \Lex^{\ell}\Big)}_{\text{Boundary dynamics}}
\end{equation}
where the generator $ \Lex^{b} $ corresponds to the bulk dynamics and  the generators $\Lex^{\ell}$ and $\Lex^{r}$ correspond to non-conservative boundary dynamics playing the role of infinitely extended reservoirs. For $ \xalpha,  \xbeta \in (0,1)$, the action of $\Lex^{b},\Lex^{\ell} $ and $ \Lex^{r}$ on functions  $f:\chi_N \to \mathbb{R}$ is given by  
\begin{equation}
\label{generator}
\begin{split}
(\Lex^{b} f)(\eta) &= \cfrac{1}{2} \, \sum_{x,y \in \Lambda_N} p(y-x) [ f(\eta^{x,y}) -f(\eta)],\\
(\Lex^{\ell} f)(\eta) &=  \sum_{x \in \Lambda_N} r_N^-(\tfrac xN) c_{x}(\eta; \xalpha)  [f(\eta^x) - f(\eta)],\\
(\Lex^{r} f)(\eta)&= \sum_{x \in \Lambda_N} r_N^+(\tfrac xN)c_{x}(\eta; \xbeta) [f(\eta^x) - f(\eta)],
\end{split}
\end{equation}
where $p(\cdot)$ is given by \eqref{eq:prob} and $r_N^{\pm}$ are given by \eqref{eq:functions_r}. Above the configurations $\eta^{x,y}$ and $\eta^x$ are defined by
$$
\eta^{x,y}(z) = \eta(y)\mathbbm{1}_{z=x}+\eta(x)\mathbbm{1}_{z=y}+ \eta(z)\mathbbm{1}_{z\neq x,y}\quad \textrm{and}\quad  \eta^{x}(z) = (1-\eta(x))\mathbbm{1}_{z=x}+ \eta(z)\mathbbm{1}_{z\neq x}$$
and for any $x\in \Lambda_{N}$,   any $\eta \in \chi_{N}$  and $\rho \in \{\xalpha,\xbeta\}$ we have that 
\begin{equation}\label{c_rate}
c_{x}(\eta;\rho)=[ \eta(x) (1-\rho) + (1-\eta(x)) \rho].
\end{equation}

When $\kappa=0$, i.e. when the boundary reservoirs are not present, the exclusion process with long jumps conserves the number of particles $\sum_{x \in \Lambda_N} \eta (x)$ and  thus, as for the zero-range,  it is  ergodic when restricted to  the set of configurations with a fixed  number of particles. As a consequence, when $\kappa=0$, it has a one parameter family of invariant measures, which are the Bernoulli product measures with parameter $ \rho \in [0,1]$ denoted by $\{\mu_\rho\}_\rho$.

\subsection{Non-equilibrium stationary states (NESS)}

\subsubsection {NESS of the boundary driven zero range process}

We prove below that there exists a (unique) invariant measure, denoted by $\nu_{ss}^N$, for the boundary driven zero-range process $(\xi_t)_{t\ge 0}$. A remarkable fact is that this NESS has a product form.

Given a function $\vf_{N}: \Lambda_{N}\to [0, \varphi^*)$, we define on $\Omega_{N}$ the product probability measure $\tilde\nu_{N}:= \tilde\nu_{\vf_{N}} $ with marginal distributions given by  
\begin{equation}
\tilde\nu_{\varphi_N (\cdot)}\lbrace \xi \in \Omega_{N} : \xi(x) = k \rbrace = \dfrac{\left(\vf_{N}(x)\right)^{k}}{Z(\vf_{N}(x))g(k)!}. 
\end{equation}

\begin{prop} \label{Invariant measure zr}
For $\zalpha, \zbeta \in (0,m^{*})$ with $\alpha \le \beta$, there exists a unique function  $$\vf_{N} :=\vf_{N}({\Phi(\zalpha),\Phi(\zbeta)}):\Lambda_N \to [\Phi(\alpha), \Phi(\beta)] \subset (0, \varphi^*)$$
{solving}  the traffic equation 
\begin{equation}
\label{Traffic equation}
\begin{split}
\vf_{N}(x)\Big(\sum_{y\in \Lambda_{N}}p(y-x)&+\kappa N^{-\theta}( r^{+}_{N}(\tfrac{x}{N})+r^{-}_{N}(\tfrac{x}{N})\big)\Big)\\
=& \sum_{y\in\Lambda_{N}}\vf_{N}(y)p(y-x)+ \frac{\kappa}{ N^{\theta}}\left(\Phi(\zbeta) r_{N}^{+}\left(\tfrac{x}{N}\right)+ \Phi(\zalpha) r_{N}^{-}\left(\tfrac{x}{N}\right)\right).
\end{split}
\end{equation} 
The product probability measure $\tilde\nu_{\varphi_N(\cdot)}$ associated to this profile $\varphi_N$ coincides with the NESS  $\nu_{ss}^N$ of the boundary driven zero-range process.
\end{prop}

\begin{rem}
The assumption $\zalpha, \zbeta \in (0,m^{*})$ is crucial to establish this result. Otherwise, condensation appears and the invariant measure does not exist since mass is growing with time at the boundaries. We refer the reader to \cite{LMS} for more information in the case of the boundary driven zero-range process with nearest-neighbor jumps.   
\end{rem}

\subsubsection{NESS for the exclusion process with long jumps}\label{exc_sec}

The  boundary driven exclusion process generated by $\Lex$ (see \eqref{genexcl}) has a unique invariant measure that we will denote by $\ex$. If $\xalpha = \xbeta =\rho$  then $\ex= \mu_{\rho}$, the Bernoulli product measure with parameter $ \rho \in [0,1]$. Differently from the NESS $\zr$ of the boundary driven zero-range process, the non-equilibrium stationary state $\ex$ of the exclusion process is not product and no explicit form is known. In \cite{cpb,BGBJ,BJ,BPS} some macroscopic information on $\ex$ has been obtained and we refer the interested reader to  Appendix \ref{app:hydrostaticsexclusion}.

{The following proposition establishes the equality between the average of a certain observable of the zero range process in its NESS and the density of the exclusion process in its NESS.}

\begin{prop}{\label{pro_exzr}}
{Consider the boundary driven zero-range  whose generator is defined by \eqref{eq:generatorZR} and the exclusion process whose generator is defined in  \eqref{generator} with}
\begin{equation}\label{tildealphabeta}
{\tilde \alpha}=\frac{\Phi (\alpha)}{\Phi (\alpha)+\Phi(\beta)}, \quad {\tilde \beta}=\frac{\Phi (\beta)}{\Phi(\alpha)+\Phi(\beta)}.
\end{equation}
Then, for all $x\in \Lambda_N$, 
\begin{equation}
\label{eq:centralequation}
    {E}_{\zr} [g(\xi (x))] =
    \left(\Phi (\alpha)+ \Phi (\beta)\right)\, {E}_{\ex} [\eta (x)] =\varphi_N (x),
\end{equation}
where $\varphi_N$ is the solution of the traffic equation \eqref{Traffic equation}.
\end{prop}

\begin{rem}
We observe that there exists, in the case of periodic boundary conditions, also a {\bf dynamical} mapping between the zero-range process and the exclusion process, which was first introduced in \cite{Kipnis} for a special choice of the rate $g(k)=\textbf{1}_{\{k> 0\}}$. In that mapping, the number of particles in the zero-range process becomes the number of holes between consecutive particles in the exclusion process. This mapping holds only if the zero-range dynamics is with nearest-neighbors jumps because it is crucial to have an order of the  particles in the exclusion dynamics, which is also with nearest-neighbors jumps. In \cite{evans0} the author extended the previous mapping to the case where $g$ is general (even inhomogeneous) but still with nearest-neighbors jumps. In the associated exclusion dynamics the rate to jump to the neighbor depends on the number of holes between the jumping particle and the next particle in the direction of the jump.
The relationship between the two processes that we derive in this article is very different and less powerful: it is a {\bf static} mapping in the sense of expectations as given in \eqref{eq:centralequation} and not a {\bf dynamical} mapping, which holds only for nearest-neighbors dynamics. 
\end{rem}


\subsection{Results}

\subsubsection{Fractional hydrostatics}

In this  section we obtain  the  hydrostatic limit of the boundary driven ZRP, i.e. we derive the form of the macroscopic density profile of the boundary driven ZRP in its NESS. In order to do it we need to introduce the definition of (weak) solutions to the hydrostatic equations of the boundary driven exclusion process with long jumps. First we need to define some sets of test functions. To that end, for $m\in\mathbb N$,  let  $C^{m}([0,1])$   (resp. $C_{c}^{m}((0,1))$)  be the set of all $m$ continuously differentiable real-valued  functions defined on $[0,1]$ (resp. and with compact support contained in $(0,1)$). We also use the notation $\langle\cdot,\cdot\rangle$ for the inner product in $L^2([0,1])$  and the corresponding norm is denoted by $\| \cdot \|.$ Let us now introduce the operators involved in the equations and the fractional Sobolev spaces that we will deal with. 

The regional fractional Laplacian $\mathbb L$ on the interval $[0,1]$ is the operator acting on functions $f:[0,1] \rightarrow \RR$ such that
\begin{equation*}
\int_{0}^1 \cfrac{|f(u)|}{(1 +|u|)^{1+\gamma}} du < \infty
\end{equation*}
as
\begin{equation}
\label{definition}
({\bb L} f)(u) = c_\gamma  \lim_{\epsilon \to 0} \int_0^1 {\bb 1}_{|u-v| \ge \epsilon} \cfrac{f(v) -f(u)}{|u-v|^{1+\gamma}} dv,
\end{equation}
for any $u \in [0,1]$ if  the limit exists. We note that $\mathbb L f$ is well defined, if, for example,  $f \in C^2([0,1])$. We also  introduce the semi inner-product $\langle\cdot,\cdot\rangle_{\gamma/2}$, and the corresponding semi-norm $||\cdot||_{\gamma/2}= \langle \cdot, \cdot \rangle_{\gamma/2}$, defined by
\begin{equation}
\langle f, g \rangle_{\gamma/2} =  \cfrac{c_{\gamma}}{2} \iint_{[0,1]^2} \cfrac{(f(u) -f(v)) (g(u) -g(v))}{|u-v|^{1+\gamma}} du dv,
\end{equation}  
where $f,g:[0,1] \rightarrow \RR$ are functions such that $\| f\|_{\gamma/2}<\infty$ and $\|g\|_{\gamma/2}<\infty$.

\begin{definition}
	\label{DefSobolevspace}
	Let  $\mathcal{H}^{\gamma/2}:=\mathcal{H}^{\gamma/2}([0,1])$ be the Sobolev space containing all the functions $g\in L^2([0,1])$ such that $\| g \|_{\gamma/2} <\infty$, which is a Hilbert space endowed with the norm $\|\cdot\|_{{\mc H}^{\gamma/2}}$ defined  by
	$$\| g \|_{\mathcal{H}^{\gamma/2}}^{2}:= \| g \|^{2} + \| g \|_{\gamma/2}^{2} .$$
	If $\gamma \in (1,2)$, by Theorem 8.2. of \cite{DNPV}, its elements coincide a.e. with continuous functions on $[0,1]$. 
	\end{definition}

Recall \eqref{tildealphabeta}. We define  two functions $V_0, V_1:(0,1)\to (0,\infty)$ by  
\begin{equation}\label{eq:useful}
V_1(u)=r^-(u)+r^+(u)\quad \textrm{and}\quad V_0(u)= \tilde\zalpha r^-(u)+\ \tilde\zbeta r^+(u)\end{equation}  where the functions $r^{\pm}: (0,1) \to (0, \infty)$ are given by
\begin{equation} \label{def:rpm}
r^- (u)=c_\gamma \gamma^{-1} u^{-\gamma},\quad  r^+ (u) = c_\gamma \gamma^{-1} (1-u)^{-\gamma}.
\end{equation}

\medskip 

We present now the different macroscopic equations which will appear in our study. The proof of the hydrostatic limit require to formulate these equations in a weak sense, i.e. in a distributional sense.

\begin{definition}
Let $\gamma \in (0,2)$ and $\hat \kappa>0$. We say that $\rho:[0,1]\to [0,1]$ is a weak solution of the stationary regional fractional reaction-diffusion equation with non-homogeneous Dirichlet boundary conditions given by
 \begin{equation}\label{DRex_ZR}
\begin{cases}
\mathbb{L}\rho(u)  +\hat \kappa(  V_{0}(u )- V_{1}(u)\rho(u))=0, \quad \forall u\in (0,1) ,\\
\rho(0)= \tilde\alpha,\quad \rho(1)= \tilde\beta,
\end{cases}
\end{equation}
if 
\begin{itemize}
    \item [a)] $ \rho \in \mathcal{H}^{\gamma/2}$. 
    \item [b)] $\displaystyle\int \dfrac{(\tilde\alpha-\rho(u))^2}{u^\gamma} +\dfrac{(\tilde\beta-\rho(u))^2}{(1-u)^\gamma}du<\infty.$
    \item [c)] For all $G\in C_c^\infty((0,1))$  we have that 
    $F_{RD}(\rho, G) := \langle\rho,\mathbb{L}G\rangle +\hat \kappa (\langle G,V_0\rangle)-\langle\rho, GV_1\rangle) = 0.$
\end{itemize}
\end{definition}

\begin{definition}
Let $\gamma \in (1,2)$. We say that $\rho:[0,1]\to [0,1]$ is a weak solution of the stationary regional fractional diffusion equation 
 with  non-homogeneous Dirichlet boundary condition
 given by
 \begin{equation}\label{Dex_ZR}
\begin{cases}
\mathbb{L}\rho(u) =0, \quad \forall u\in (0,1),\\
\rho(0)= \tilde\alpha,\quad 
\rho(1)= \tilde\beta,
\end{cases}
\end{equation}
if 
\begin{itemize}
    \item [a)] $ \rho \in \mathcal{H}^{\gamma/2}$.
    \item [b)] For all {$G\in C_c^\infty([0,1])$} we have that 
   $ F_{Dir}(\rho, G) := \langle\rho,\mathbb{L}G\rangle  = 0.$
    \item [c)] $\rho(0)=\tilde\alpha$ and $\rho(1)=\tilde\beta$.
\end{itemize}
\end{definition}

\begin{rem}
Since $\gamma \in (1,2)$, as mentioned previously in Definition \ref{DefSobolevspace}, if $\rho \in {\mathcal H}^{\gamma/2}$, then it coincides a.e. with a continuous function on $[0,1]$ so that item c) in last definition makes sense.
\end{rem}



\begin{definition} 
Let $\gamma \in (1,2)$ and $\hat\kappa > 0$. We say that $\rho:[0,1]\to [0,1]$ is a weak solution of the stationary regional fractional diffusion equation 
 with  fractional Robin boundary conditions
 
 \begin{equation}
 \label{Rex_ZR}
\begin{cases}
\mathbb{L}\rho(u) =0, \quad \forall u\in (0,1),\\
\chi_{\gamma}(D^\gamma\rho)(0)= {\hat \kappa}  (\tilde\alpha-\rho(0)),\\
\chi_{\gamma}(D^\gamma\rho)(1)= {\hat \kappa} (\tilde\beta-\rho(1)),
\end{cases}
\end{equation}
if 
\begin{itemize}
    \item [a)]  $ \rho \in \mathcal{H}^{\gamma/2}$.
    \item [b)]For all $G\in C^\infty([0,1])$ we have that 
   $$ F_{Rob}(\rho, G,\hat \kappa) := \langle\rho,\mathbb{L}G\rangle -  \hat\kappa \left( G(0)(\tilde\alpha-\rho(0)) + G(1)(\tilde\beta-\rho(1)\right) = 0.$$
\end{itemize}
\end{definition}
Above
\begin{equation}\label{eq:frac_deriv}
(D^{\gamma}\rho)(0) = \lim _{u\to 0^{+}} \rho'(u)u^{2-\gamma}\quad \textrm{and}\quad (D^{\gamma}\rho)(0) = \lim _{u\to 1^{-}} \rho'(u)(1-u)^{2-\gamma}\end{equation}
and $\chi_{\gamma}$ is a constant defined below equation (7.4) in \cite{GM}. 


\begin{definition} 
Let $\gamma \in (0,2)$ and $\hat M \in [0,1]$. We say that $\rho:[0,1]\to [0,1]$ is a weak solution of the stationary regional fractional diffusion equation 
 with  fractional Neumann boundary conditions and total mass $\hat M$ 
 
 \begin{equation}
 \label{Nex_ZR}
\begin{cases}
\mathbb{L}\rho(u) =0, \quad \forall u\in (0,1),\\
(D^\gamma\rho)(0)=0,\\
(D^\gamma\rho)(1)=0,
\end{cases}, \quad \int_0^1\rho(u)\,du={\hat M}.
\end{equation}
if 
\begin{itemize}
    \item [a)]  $ \rho \in \mathcal{H}^{\gamma/2}$.
    \item [b)]For all $G\in C^\infty([0,1])$ we have that 
   $$ F_{Neu}(\rho, G) := \langle\rho,\mathbb{L}G\rangle = 0.$$
{\item   [c)] $\int_0^1\rho(u)\,du={\hat M}.$}
\end{itemize}
\end{definition}

\begin{rem}
Each of the previous equations has the property that their weak solutions are unique and this is a crucial feature for the proof of the  hydrostatic limit. The uniqueness result is proved in detail in \cite{cpb,BJ,BPS}, apart the Neumann case, which we present below.
\end{rem}
\begin{lem}
\label{lem:214}
The weak solution of \eqref{Nex_ZR} in the sense given above is unique and equal to $\hat M$.
\end{lem}

\begin{proof}
Observe first that the constant function $\hat M$ is a solution. Let us consider two weak solutions $\rho^1,\rho^2$ of \eqref{Nex_ZR} with $\hat\kappa=0$ and let $\bar\rho=\rho^1-\rho^2$.  Let us assume first that one can take $G=\bar\rho$ in $F_{Neu}(\bar\rho,G) $ to get  $ F_{Neu}(\bar\rho, \bar\rho) = \langle\bar\rho,\mathbb{L}\bar\rho\rangle  = 0.$ From the integration by parts formula, see, for example, Proposition 2.1 in \cite{BPS}, we get that $\langle \bar\rho, \bar\rho\rangle_{\gamma/2}$=0, which implies that $\bar\rho$ is constant almost everywhere. From item c) of the definition of weak solution, we conclude that $\bar\rho(u)=0$ for almost everywhere $u\in[0,1]$. Now, we just have to redo the argument by considering a sequence $\{G_k\}_{k\geq 1}$ of functions in $C^\infty([0,1])$ converging to $\bar\rho$ with respect to the norm $\|\cdot\|_{{\mc H}^{\gamma/2}}$, and the proof ends.
\end{proof}

Recall that $d>0$ is the parameter defined in \eqref{eq:d} and that by Proposition \ref{Invariant measure zr} the function $\varphi_N$ takes values in $[\Phi(\alpha), \Phi(\beta)]$. The following theorem is a form of hydrostatic limit for the non-equilibrium stationary boundary driven exclusion process with long jumps.

\begin{thm}\label{thm:gen_conv}
For any continuous function $G:[0,1]\to \RR$  and any function $F:[\Phi(\alpha), \Phi(\beta)] \times [0,1]\rightarrow \RR$ which is Lipschitz in the first component, we have that 
$$\lim _{N\to\infty} \left\vert {\dfrac{1}{\# \Lambda_N}} \sum_{x\in \Lambda_{N} }  G(\tfrac{x}{N}) F(\varphi_{N}(x),\tfrac{x}{N}) - \int _{0} ^{1} G(u)F\left( \left[ \Phi(\alpha)+ \Phi (\beta) \right]\bar\rho(u),u \right)du\right\vert = 0$$
where $\bar \rho: [0,1] \to [0,1]$ is the measurable function defined by  
\begin{itemize}
\item [a)] for $\theta <0$ and $\gamma\in (0,2)$,  
$
{\bar  \rho} (u) = \dfrac{V_{0}(u)}{V_{1}(u)}.
$

\item  [b)] for $\theta =0$ and $\gamma\in (0,2)-\{1\}$,  $\bar \rho(\cdot)$ is the unique weak  solution of  \eqref{DRex_ZR}, with $\hat \kappa=\kappa$.

\item  [c)] for $\theta \in (0,\gamma-1)$ and $\gamma\in (1,2)$, $\bar \rho(\cdot)$ is the unique  weak solution of \eqref{Dex_ZR}.

 \item  [d)] for  $\theta = \gamma-1$ and $\gamma\in (1,2)$, $\bar \rho(\cdot)$ is the weak  solution of \eqref{Rex_ZR} with {$\hat \kappa=\kappa d$}.

\item  [e)]  for $\theta > 0$ and  {$\gamma\in (0, 1]$} or for $\theta >  \gamma- 1$ and  $\gamma\in (1, 2)$, $\bar \rho(\cdot)$ is the weak solution of \eqref{Rex_ZR} with {$\hat \kappa =0$}.
\end{itemize}

Moreover the profile 
$\bar \rho (\cdot)$ takes values in $[\tilde \alpha, \tilde \beta]$.
\end{thm}
{Above,  $\#$ is  the counting measure. }
\begin{figure}[h]
    \centering
    \includegraphics{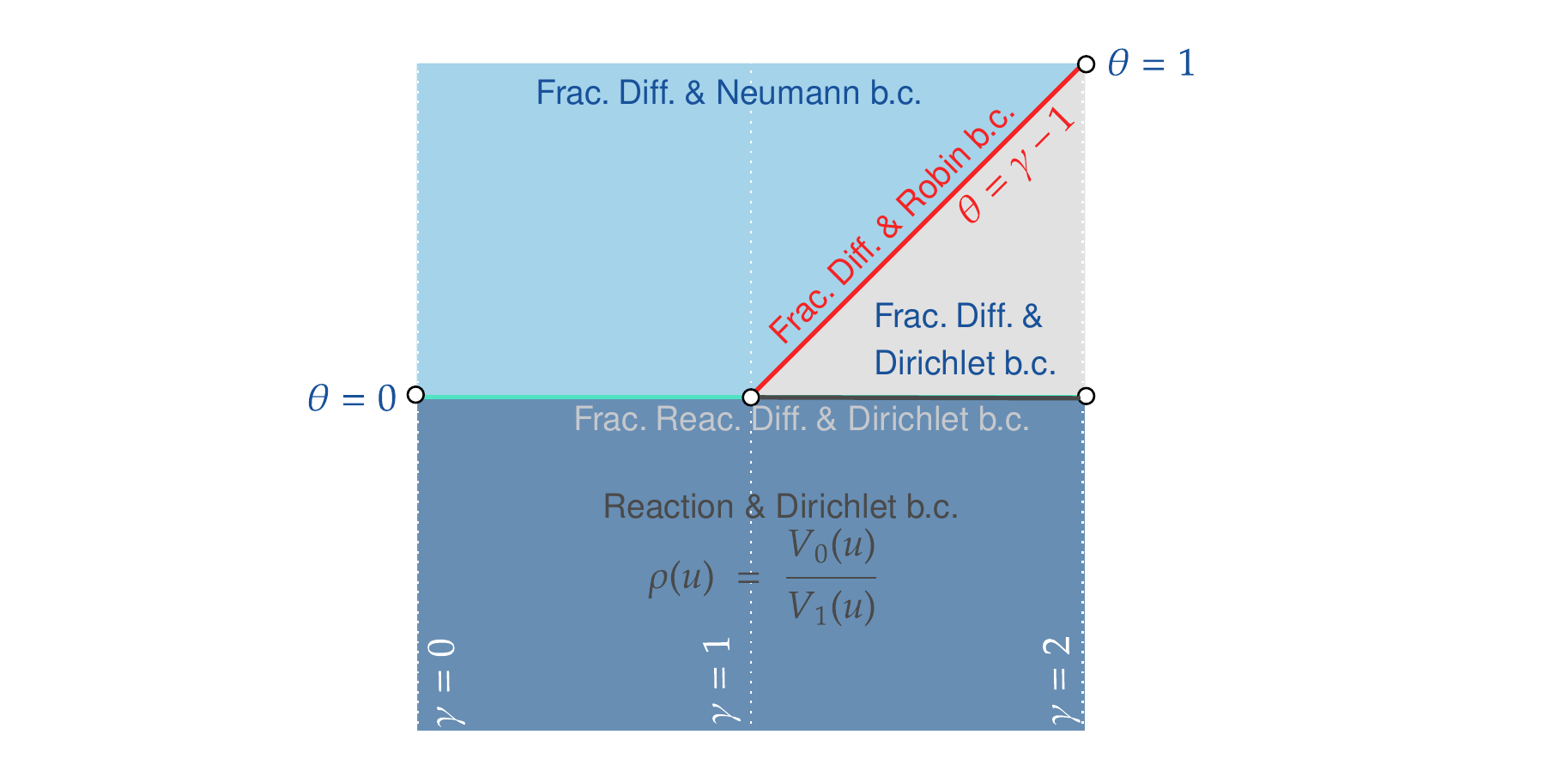}\caption{Hydrostatic behavior depending on the values of $ \theta$ (vertical axis) and  $\gamma$ (horizontal axis). }
\end{figure}
Last theorem, proved in Section \ref{sec:4}, permits to prove  the hydrostatics, in mean{\footnote{{A sequence of integrable random variables $(X_N)_{N}$ is said to converge in mean if $\{\mathbb E [X_N)]\}_{N}$ converges. }}}, of the boundary driven ZRP.

\begin{cor}[Hydrostatic limit in mean]
\label{cor-hydrostat} 
\quad

For any continuous function $G:[0,1]\to \RR$, we have that
$$\lim _{N\to\infty} \left\vert {\dfrac{1}{\# \Lambda_N}} \sum_{x\in \Lambda_{N} }  G(\tfrac{x}{N}) {E}_{\nu_{ss}^N} \left[ \xi (x) \right]  - \int _{0} ^{1} G(u) {\bar m}(u)  du\right\vert = 0$$
where the hydrostatic profile $\bar m(\cdot)$ of the boundary driven zero-range process with long jumps is defined by
\begin{equation}
\label{eq:mprofile}
\forall u \in [0,1], \quad {\bar m}(u)=R\left[\left( \Phi (\alpha) + \Phi (\beta)\right) {\bar \rho} (u) \right]
\end{equation}
and  $\bar \rho(\cdot)$ is the hydrostatic profile of the boundary driven exclusion process with long jumps given in Theorem \ref{thm:gen_conv}.
\end{cor}
\begin{proof}
Recall first that because of the form of $\nu_{ss}^N$ we have that 
\begin{equation}
\label{eq:007}
    {E}_{\nu_{ss}^N} [\xi(x)] =R (\vf_N (x))
\end{equation}
for any $x\in \Lambda_N$. We apply Theorem \ref{thm:gen_conv}  with the function $F:[\Phi(\alpha), \Phi(\beta)] \times [0,1] \to \mathbb R$ defined by 
$$F(\varphi,u) = R(\varphi).$$
This function is Lipschitz in the first variable since $R$ is analytic on $(0, \varphi^*)$ and $[\Phi(\alpha), \Phi(\beta)] \subset (0, \varphi^*)$. Hence, for any continuous function $G:[0,1]\to \RR$ we have that 
$$\lim _{N\to\infty} \left\vert {\dfrac{1}{\# \Lambda_N}} \sum_{x\in \Lambda_{N} }  G(\tfrac{x}{N}) R(\varphi_{N}(x)) - \int _{0} ^{1} G(u)R\left( \left[ \Phi(\alpha)+ \Phi (\beta) \right]\bar\rho(u)\right)du\right\vert = 0.$$
From \eqref{eq:007} we conclude the proof.
\end{proof}

\begin{center}
\begin{figure}[htb!]\label{fig:statprofZR}
    \centering
    \includegraphics[scale =0.2]{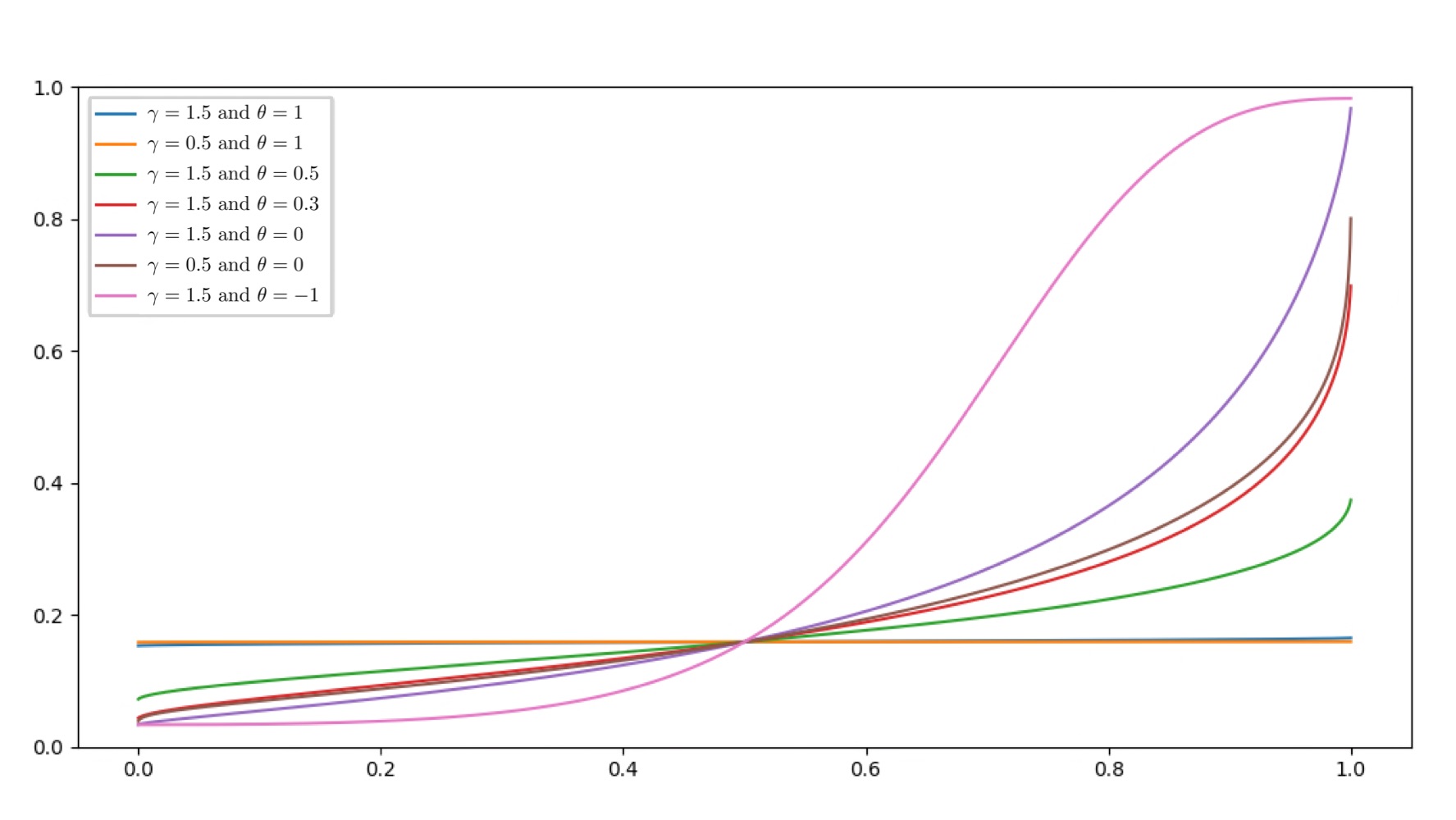}
    \caption{Stationary profiles of the boundary driven zero-range process with long-jumps, with $\alpha = 0.2$, $\beta=0.8$, $g(k)=\left(1+\frac{3}{k}\right)^3$ if $k\neq 0$ and $g(0)=0$. All the profiles $\bar m$ are such that ${\bar m} (1/2)=R \left( \tfrac{\Phi (\alpha)+\Phi(\beta)}{2}\right)$ (which follows from \eqref{eq:symmetry} evaluated for $x=N/2$).}  
\end{figure}
\end{center}

\begin{rem}
In Figure 2 it is plotted the profile $\bar m(\cdot)$. Observe that in some range of the parameters ({$\gamma< 2$}, $\theta\ge 0$) the profile is non-differentiable at the boundaries. An open question is to determine its exact behaviour there.
\end{rem}

\subsubsection{Fractional Fick's law}
Our second result is the following ``fractional Fick's law''. For  $x \in \Lambda_{N}\cup{\{N\}}$ and a configuration $\xi$,  we denote the current over the value $x-\tfrac{1}{2}$ by  $W_{x}^{zr}(\xi)$ and we define it as the rate of particles crossing  $x-\tfrac{1}{2}$ from  left to  right minus the rate of particles crossing $x-\tfrac{1}{2}$ from  right to  left. Therefore, the current  can be written as 
\begin{equation}\label{def:current}
\begin{split}
W_x^{zr}(\xi)=  &\sum_{\substack {1 \le y\le x-1 \\
 x-1<z\le N-1}}  p(z-y) (g(\xi(y))  -g(\xi(z)) ) \\
+&{\frac{\kappa}{ N^{\theta}}}\Bigg[\sum_{x \le z \le N-1} r_N^-(\tfrac zN) ( \Phi(\zalpha) -g(\xi(z))) -\sum_{1 \le y\le x-1}r_N^+(\tfrac yN)   (\Phi(\zbeta)-g(\xi(y)))\Bigg].
\end{split}
\end{equation}
Moreover, changing in the last definition $g(\xi(\cdot))$ by $\eta (\cdot)$, $\Phi(\zalpha)$ by  $\xalpha $  and $\Phi(\zbeta)$ by  $\xbeta $ we obtain the  definition of  the current for the exclusion process (see \eqref{eq:currentexclusion0}).  We will denote  the current for the exclusion process by $W_x^{ex}(\eta)$.
From Proposition \ref{pro_exzr} it is not difficult to see that 
$$E_{\zr }[{W_{x}^{zr}}] = \left( \Phi(\alpha) +\Phi (\beta) \right) \ E_{\ex }[{W_{x}^{ex}}].$$
Therefore, it is sufficient to study the behaviour of the average current for the boundary driven exclusion process with long jumps. From Theorem  \ref{thm:Fick} we can derive the next result.
\begin{thm}[Fractional Fick's law]
\label{thm:Fick}
Let $\bar m(\cdot)$ be the hydrostatic profile of the boundary driven ZRP defined in Corollary \ref{cor-hydrostat}.
For $u\in (0,1)$ the following fractional Fick's law holds, apart from the case $\theta=0$ and $\gamma=1$:
\begin{itemize}
    \item [a)] for $\theta < 0$, 
 \begin{equation}
 \label{eq:fl67}
 \begin{split}
 \lim_{N \to \infty} \frac{1}{N^{1-\theta-\gamma}} E_{ \zr }[{W_{[uN]}^{zr}}] &= \kappa\int_{u}^{1}(\Phi(\zalpha) -\Phi (\bar m(v)))r^{-}(v)dv- \kappa\int_{0}^{u}(\Phi(\zbeta) -\Phi (\bar m (v)) r^{+}(v)dv\\
&=\kappa c_\gamma \gamma^{-1} \int_0^1\dfrac{\Phi(\zalpha)-\Phi(\zbeta)}{v^\gamma+(1-v)^\gamma} dv;
 \end{split}
 \end{equation}
\item [b)]
for  $\theta = 0$,
\begin{equation}
 \label{eq:fl68}
 \begin{split}
 \lim_{N \to \infty}\frac{1}{N^{1-\gamma}} E_{\zr }[{W_{[uN]}^{zr}}] &=c_{\gamma} \int_{0}^u \; \int_{u}^{1} \, \cfrac{{ \Phi} (\bar m(v)) -{ \Phi} (\bar m(w))}{(w-v)^{1+\gamma}}\, dw dv + \kappa\int_{u}^{1}(\Phi(\zalpha) -\Phi(\bar m(v)))r^{-}(v)dv\\
&- \kappa\int_{0}^{u}(\Phi(\zbeta) -\Phi(\bar  m(v)))r^{+}(v)dv;
 \end{split}
 \end{equation}
\item [c)] for $\theta>0$,
\begin{equation}
 \label{eq:fl69}
 \begin{split}
 \lim_{N \to \infty} \frac{1}{N^{1-\gamma}} E_{\zr }[{W_{[uN]}^{zr}}] &=c_{\gamma} \int_{0}^u \; \int_{u}^{1} \, \cfrac{{ \Phi} (\bar m(v)) -{ \Phi} (\bar m(w))}{(w-v)^{1+\gamma}}\, dw dv. 
 \end{split}
 \end{equation}
 \end{itemize}
 \end{thm}

\begin{rem}
In the ``Neumann case'', i.e. $\gamma \in (1,2)$ and $\theta>\gamma-1$ or {$\gamma \in (0,1]$} and  $\theta>0$, $\bar m$ is constant and the current vanishes as expected.  
\end{rem}

\subsubsection{Static large deviations}
We want  to obtain the large deviation principle associated to the hydrostatic results. More precisely, we want to estimate the probability of a deviation from the typical profile which satisfies the hydrostatic equation,  but remains close to some prescribed path. In order to do this, we consider a perturbation  of the system. First we need to introduce some notation.

Let ${\mc M}$ be the space of finite signed Borel measures on $[0,1]$. It is known that ${\mc M}$ is the topological dual of $C^0([0,1])$, when the latter is equipped with the uniform convergence. Then ${\mathcal M}$ equipped with the weak-$\star$ topology {\footnote{We recall that a sequence $\{\mu_n\}_n \in {\mathcal M}$ converges $\star$-weakly to $\mu \in {\mathcal M}$ if, and only if, for all  $G\in C([0,1])$ we have that $\int_0^1 G(x) d\mu_n (x)$ converges to $\int_0^1 G(x) d\mu (x)$. This coincides with the notion of `weak convergence' used in probability theory.}} is a Banach space. Let ${\mc M}^+ \subset {\mathcal M}$ be the cone of positive measures.  For any $\xi \in \Omega_N$ the empirical measure  ${\pi}^N  (\xi, du) \in {\mc M}^+$ is defined by 
\begin{equation}
{\pi}^N  (\xi, du) ={\dfrac{1}{\# \Lambda_N}} \sum_{x\in \Lambda_N}\xi_x   \delta_{\tfrac{x}{N}}(du)
\end{equation}
where $\delta_u$ is the Dirac mass on  $u \in [0,1]$. We assume that $\xi$ is distributed according to $\nu_{ss}^N$ and to simplify we denote  ${\pi}^N (\xi, du)$ by ${ \pi}^N (du)$.  The action of $\pi^N \in {\mc M}^+$ on a continuous function $G:[0,1] \to \RR$ is denoted by 
$$\langle \pi^N , G \rangle : = \int_{[0,1]} G(u) \pi^N (du)= {\dfrac{1}{\# \Lambda_N}} \sum_{x\in\Lambda_N}  G(\tfrac xN)\xi_x.$$  
We also define the  functional $\Lambda: C^0 ([0,1])\to \mathbb{R}$ by
$$\forall G \in C^0 ([0,1]), \quad \Lambda(G) =\int_{0}^{1}\log\left( \frac{ Z(e^{G(u)}\Phi (\bar m(u)))}{Z(\Phi (\bar m(u)))}\right) du$$
where $\bar m$ is the hydrostatic profile defined in Corollary \ref{cor-hydrostat}.  Its Legendre transform $\Lambda^*:{\mathcal M} \to {\mathbb R}$ is given by 
\begin{equation*}
    \forall \pi \in {\mathcal M}, \quad \Lambda^* (\pi) = \sup_{G\in C^0 ([0,1])} \left\{ \langle \pi, G\rangle - \Lambda (G) \right\}.
\end{equation*}

 The functional $\Lambda^*$ can be computed more explicitly. If $\pi$ is not absolutely continuous with respect to the Lebesgue measure then it is easy to show that $\Lambda^* (\pi)=\infty$. If $\pi (du) := \pi (u) du $ is absolutely continuous with respect to the Lebesgue measure then,  from \eqref{eq:Rprime} and the fact that $\Phi(\cdot)$ is the inverse function of $R(\cdot)$, we get that
 \begin{equation*}
     \Lambda^* (\pi) = \int_0^1 \left\{ \pi(u) \log \left(\frac{\Phi(\pi(u))}{\Phi(\bar m(u))}\right) -\log \left(\frac{Z(\Phi(\pi(u)))}{Z(\Phi(\bar m(u)))} \right)  \right\}\ du.
 \end{equation*}

\begin{thm}[Large deviations]
\label{th:ldp}
If  $\varphi^*=+\infty$, then the sequence of random variables $\left\{\pi^{N}\right\}_{N\ge 1}$ satisfies a Large Deviation Principle at speed $N$ with the good rate function $\Lambda^{*}$.
\end{thm}

\section{Proof of Propositions \ref{Invariant measure zr} and \ref{pro_exzr}} \label{sec:3}

\begin{proof} [Proof of Proposition \ref{Invariant measure zr}] Observe that  \eqref{Traffic equation} is a finite-dimensional linear equation which can be written in a matrix form. Let us now introduce some notation. We define  the vector $\vf_{N}=[\vf_{N}(x)]_{x\in\Lambda_{N}}$, the square matrix of size $N-1$ denoted by  $P_{N} = [p(y-x)]_{(x,y)\in \Lambda_{N}^{2}}$, the diagonal matrix $D_N$ of size $N-1$ whose diagonal elements are given by $$ \left[\sum_{y\in \Lambda_{N}}p(y-x)+\kappa N^{-\theta}\left( r^{+}_{N}(\tfrac{x}{N})+r^{-}_{N}(\tfrac{x}{N})\right)\right]_{x\in\Lambda_{N}}, $$ and finally, the vector $R_{N} = \left[\Phi(\zbeta) r_{N}^{+}(\tfrac{x}{N})+\Phi(\zalpha) r_{N}^{-}(\tfrac{x}{N})\right]_{x\in \Lambda_{N}}$.
With this notation we can rewrite the traffic equation \eqref{Traffic equation} as 
$$(D_N -P_N) {\vf_N} = R_N.$$
Since $$\left\vert \sum_{y\in \Lambda_{N}}p(y-x)+\kappa N^{-\theta}\left( r^{+}_{N}(\tfrac{x}{N})+r^{-}_{N}(\tfrac{x}{N})\right)\right\vert > \sum_{y\in \Lambda_{N}}\vert p(y-x)\vert,$$
the matrix  $(D_{N}- P_{N})$  is strictly diagonally dominant, hence it is invertible. From this, we get that the traffic equation \eqref{Traffic equation} has a unique solution.\\

We have now to show that $\varphi_N$ takes  values in $[\Phi(\alpha), \Phi(\beta)] \subset [0, \varphi^*)$. Let $\varphi =\max_{x\in \Lambda_N} \varphi_N (x)$ and let $x_0\in \Lambda_N$ be such that $\varphi =\varphi_N (x_0)$. Evaluating \eqref{Traffic equation} at $x=x_0$ and using the fact that for any $y \in \Lambda_N$, $\varphi_N (y) \le {\varphi} =\varphi_N (x_0)$ we get that
\begin{equation}
 {\varphi} \le \frac{\Phi(\zbeta) r_{N}^{+}\left(\tfrac{x_0}{N}\right)+ \Phi(\zalpha) r_{N}^{-}\left(\tfrac{x_0}{N}\right)}{r_{N}^{+}\left(\tfrac{x_0}{N}\right)+  r_{N}^{-}\left(\tfrac{x_0}{N}\right)}\le \Phi(\beta) < \varphi^*   
\end{equation}
where the penultimate inequality follows from the assumption $\alpha\le \beta$ and the fact that $\Phi$ is increasing, while the last inequality is due to the assumption \eqref{eq:alphabeta}. A similar argument shows that $$\min_{x\in \Lambda_N}\varphi_N (x) \ge \Phi(\alpha) > 0.$$ It follows that $\tilde\nu_{\varphi_N (\cdot)}$ is a well defined probability measure, since $Z (\varphi_N (x))$ is finite for every $x \in \Lambda_N$.\\

In order to prove that $\tilde\nu_{\varphi_N (\cdot)}$ is the invariant measure of the boundary driven zero-range process with long jumps, we have to prove that for any bounded function $f:\Omega_N \to \mathbb R$, it holds
$$\int_{\Omega_{N}}(L_N f)(\xi)\, d\tilde\nu_{\varphi_N (\cdot)}(\xi)=0.$$
Observe that from the  definition of the generator we have that  
\begin{equation*}\begin{split}
 \int_{\Omega_{N}} (\Lzr f)(\xi)d\nu_{N}(\xi)
=\sum_{\xi\in\Omega_{N}}&{\tilde \nu }_{\varphi_N (\cdot)}(\xi)\left\lbrace \sum_{x\in\Lambda_{N}}\sum_{y\in\Lambda_{N}} p(y-x)g(\xi(x)) \left[ f(\xi^{x,y}) - f(\xi)\right] \right.\\\nonumber
& +   \frac{\kappa}{ N^{\theta}} \left.\sum_{x\in\Lambda_{N}}\sum_{y\geq N}p(y-x)\Phi(\zbeta) \left[ f(\xi^{y,x}) -f(\xi)\right]   +g(\xi(x)) \left[ f(\xi^{x,y}) -f(\xi)\right] \right.\\\nonumber
& +  \frac{\kappa}{ N^{\theta}} \left.\sum_{x\in\Lambda_{N}}\sum_{y\leq 0} p(y-x)\Phi(\zalpha) \left[ f(\xi^{y,x}) -f(\xi)\right]+g(\xi(x)) \left[ f(\xi^{x,y}) -f(\xi)\right]\right\rbrace.
\end{split}\end{equation*}
Performing the  change of variables $\xi \to \xi^{x,y}$, we get
\begin{equation*}\begin{split}
\int_{\Omega_{N}} &(\Lzr f)(\xi)d{\tilde \nu}_{\varphi_N (\cdot) }(\xi)\\&=\sum_{\xi\in\Omega_{N}}f(\xi)\left\lbrace \sum_{x\in\Lambda_{N}}\sum_{y\in\Lambda_{N}} p(y-x)\left[g(\xi(x)+1)  {\tilde \nu}_{\varphi_N (\cdot)}(\xi^{y,x}) -g(\xi(x)) {\tilde \nu}_{\varphi_N (\cdot)}(\xi)\right] \right.\\\nonumber
 &\qquad\qquad\quad +\frac{\kappa}{ N^{\theta}} \left.\sum_{x\in\Lambda_{N}}\sum_{y\geq N}p(y-x) \left[g(\xi(x)+1) {\tilde \nu}_{\varphi_N (\cdot)}(\xi^{y,x})-g(\xi(x)){\tilde \nu}_{\varphi_N (\cdot)}(\xi)\right]  \right.\\\nonumber
 &\qquad\qquad\quad+ \frac{\kappa}{ N^{\theta}} \left. \sum_{x\in\Lambda_{N}}\sum_{y\leq 0} p(y-x) \left[g(\xi(x)+1) {\tilde \nu}_{\varphi_N (\cdot)}(\xi^{y,x})-g(\xi(x)){\tilde \nu}_{\varphi_N (\cdot)}(\xi)\right]\right.\\\nonumber
 &\qquad\qquad\quad+\frac{\kappa}{ N^{\theta}} \sum_{x\in\Lambda_{N}} \sum_{y\geq N} p(y-x) \Phi(\zbeta)\left[ {\tilde \nu}_{\varphi_N (\cdot)}(\xi^{x,y})-{\tilde \nu}_{\varphi_N (\cdot)}(\xi)\right] \\\nonumber
 &\qquad\qquad\quad +\frac{\kappa}{ N^{\theta}} \left. \sum_{x\in\Lambda_{N}} \sum_{y\leq 0} p(y-x) \Phi(\zalpha)\left[ {\tilde \nu}_{\varphi_N (\cdot)}(\xi^{x,y})-{\tilde \nu}_{\varphi_N (\cdot)}(\xi)\right] \right\rbrace.
\end{split}\end{equation*}
Since ${\tilde \nu}_{\varphi_N (\cdot)}$  is a product measure we have that 
\begin{equation*}
{\tilde \nu}_{\varphi_N (\cdot)}(\xi^{y,x}) = 
\begin{cases}
\dfrac{{\tilde \nu}_{\varphi_N (\cdot)}(\xi)g(\xi(y))\vf_{N}(x)}{g(\xi(x)+1)\vf_N(y)}, \:\:\text{if}\:\: x,y \in \Lambda_{N} , \\
\\
\dfrac{{\tilde \nu}_{\varphi_N (\cdot)}(\xi)\vf_{N}(x)}{g(\xi(x)+1)}, \:\:\text{if}\:\:  x\in \Lambda_{N},\:\: y\notin \Lambda_{N}. \\
\\
\dfrac{{\tilde \nu}_{\varphi_N (\cdot)}(\xi)g(\xi(y))}{\vf_{N}(y)}, \:\:\text{if}\:\:  y\in \Lambda_{N},\:\: x\notin \Lambda_{N}. \\
\end{cases}
\end{equation*}
Therefore,  
\begin{equation*}\label{im3}\begin{split}
 \int_{\Omega_{N}} (\Lzr f)(\xi)&d{\tilde \nu}_{\varphi_N (\cdot)}(\xi)=\sum_{\xi \in \Omega_{N}}f(\xi){\tilde \nu}_{\varphi_N (\cdot)}(\xi)\left\lbrace\sum_{x\in\Lambda}\dfrac{g(\xi(x))}{\vf_{N}(x)}\times \right.\\
 &\quad \quad \times \left[ \sum_{y\in\Lambda_{N}}p(y-x)\vf_{N}(y) + \frac{\kappa}{ N^{\theta}}\left[\Phi(\zbeta) r^{+}_{N}(\tfrac{x}{N})+  \Phi(\zalpha) r^{-}_{N}(\tfrac{x}{N})\right] \right.\\
&  \qquad\qquad\left.\left. - \vf_{N}(x)\left(\sum_{y\in \Lambda_{N}}p(y-x)+\kappa N^{-\theta}\left( r^{+}_{N}(\tfrac{x}{N})+r^{-}_{N}(\tfrac{x}{N})\right)\right)\right]\right\rbrace\\
&+\frac{\kappa}{ N^{\theta}}\sum_{\xi \in \Omega_{N}}f(\xi){\tilde \nu}_{\varphi_N (\cdot)}(\xi)\left\lbrace \sum_{x\in\Lambda}\left(\vf_{N}(x)-\Phi(\zalpha)\right)r^{-}_{N}(\tfrac{x}{N})+\left(\vf_{N}(x)-\Phi(\zbeta)\right)r^{+}_{N}(\tfrac{x}{N}) \right\rbrace.
\end{split}
\end{equation*}
By using the fact that $\vf_{N}(x)$ is solution of the traffic equation (\ref{Traffic equation}) the first three lines  in  last display are  equal to $0$. It remains to see that the last line in last display is also equal to $0$. For that purpose we observe that we can rewrite it as  
\begin{eqnarray*}
&&\frac{\kappa}{ N^{\theta}}\sum_{\xi \in \Omega_{N}}f(\xi){\tilde \nu}_{\varphi_N (\cdot)}(\xi)\left\lbrace \sum_{x\in\Lambda}\left[\vf_{N}(x)+\vf_{N}(N-x)-(\Phi(\zalpha)+\Phi(\zbeta))\right]r^{-}_{N}\left(\tfrac{x}{N}\right)\right\rbrace. 
\end{eqnarray*}
Hence it is sufficient to show that 
\begin{equation}
\label{eq:symmetry}
\vf_{N}(x)+\vf_{N}(N-x) = \Phi(\zalpha) + \Phi(\zbeta).
\end{equation}
To see this it is enough to  sum  \eqref{Traffic equation} evaluated at $x$ and the same equation evaluated at $N-x$, and use the fact that 
$$\sum _{y\in \Lambda_{N}} p(y-(N-x))= \sum _{y\in \Lambda_{N}} p(y-x), \quad  r^{+}(\tfrac{N-x}{N}) = r^{-}_{N}(\tfrac{x}{N}), \quad  r^{-}(\tfrac{N-x}{N}) = r^{+}_{N}(\tfrac{x}{N})$$
and
 $$\sum _{y\in \Lambda_{N}} p(y-(N-x))\vf_N(y)=\sum _{y\in \Lambda_{N}} p(y-x)\vf_N(N-y).$$ From this we get that 
\begin{equation*}
\begin{split}
&\left(\vf_{N}(x)+\vf_{N}(N-x)\right)\left(\sum_{y\in \Lambda_{N}}p(y-x)+\kappa N^{-\theta}\left( r^{+}_{N}(\tfrac{x}{N})+r^{-}_{N}(\tfrac{x}{N})\right)\right)\\
=& \sum_{y\in\Lambda_{N}}p(y-x)\left(\vf_{N}(y)+\vf_{N}(N-y)\right)+ \kappa N^{-\theta}\left(\Phi(\zbeta) + \Phi(\zalpha)\right)\left( r_{N}^{+}\left(\tfrac{x}{N}\right)+ r_{N}^{-}\left(\tfrac{x}{N}\right)\right).
\end{split}
\end{equation*} 
We have seen above that the matrix $D_N -P_N$ is invertible,  so that this discrete equation with unknown $\psi (\cdot)=\varphi_N (\cdot) +\varphi_N (N-\cdot))$ has a unique solution. Since the constant function $\psi (\cdot)=\Phi(\zbeta)+ \Phi(\zalpha)$ is a solution, we can conclude that $\vf_{N}(x)+\vf_{N}(N-x) = \Phi(\zalpha) + \Phi(\zbeta)$
and this ends the proof.
\end{proof}

\begin{proof} [Proof of Proposition \ref{pro_exzr}] 

Note that for each $x\in \Lambda_N$ we have
\begin{equation*}\label{subst.}
\begin{split}
E_{\zr}[g(\xi(x))] &= \int _{\Omega_{N}} g(\xi(x))d \zr (\xi)
                    = \sum _{k=0}^{\infty}\dfrac{g(k)(\vf_N(x))^k}{Z(\vf_N(x))g(k)!}
    =\dfrac{\vf_N(x)}{Z(\vf_N(x))} \sum _{k=0}^{\infty}\dfrac{(\vf_N(x))^k}{g(k)!}
                          = \vf_N(x).
\end{split}
\end{equation*}

On the other hand, since $\ex$ is a stationary measure, by writing for each $x$ that $E_{\ex}\left[{\mathcal L}_N f_x \right]=0$ (recall \eqref{generator}) with $f_x(\eta) = \eta(x)$, we get directly that $\left( \Phi (\alpha)+ \Phi (\beta)\right)\, E_{\ex} [\eta (x)]$ is the solution of the traffic equation \eqref{Traffic equation}. So, by uniqueness of the solution of \eqref{Traffic equation} we conclude that
$E_{\ex} [\eta (x)] = \vf_N(x).$

\end{proof}

\section{Proof of Theorem \ref{thm:gen_conv}}\label{HyFick}
\label{sec:4}

\begin{proof}
For $x\in \Lambda _{N}$ and $\ve >0$ we define{\footnote{We assume for simplicity that $\ve N$ is an integer.}} the box of length $2\ve N + 1$ centered around $x$
$$I_{\ve N}(x) = [x-\ve N, x+\ve N] \cap \Lambda_{N}$$ and for  a function $h:\Lambda_{N}\to \RR$  we define its average in this box by 
\begin{equation}\label{eq:average}
A[h, I_{\ve N}(x)] =\dfrac{1}{\# I_{\ve N}(x)}\sum _{y\in I_{\ve N}(x)}h(y) 
\end{equation}
where $\#$ is  the counting measure. Fix $\ve > 0$.  By adding and subtracting  the term 
$${\dfrac{1}{\# \Lambda_N}} \sum_{x\in \Lambda_{N} }  G(\tfrac{x}{N}) F\left(A[\vf_N, I_{\ve N}(x)],\tfrac{x}{N} \right)$$
and using the  triangular inequality, we get that
\begin{equation*}
    \begin{split}
        \Big|{\dfrac{1}{\# \Lambda_N}} \sum_{x\in \Lambda_{N} }  &G(\tfrac{x}{N}) F(\varphi_{N}(x),\tfrac{x}{N}) - \int _{0} ^{1} G(u) F\left( \left[ \Phi(\alpha)+ \Phi (\beta) \right]\bar\rho(u),u \right) du\Big|  \\
        \leq &  \Big|{\dfrac{1}{\# \Lambda_N}} \sum_{x\in \Lambda_{N} }  G(\tfrac{x}{N})\left( F(\varphi_{N}(x),\tfrac{x}{N})- F\left(A[\vf_N, I_{\ve N}(x)],\tfrac{x}{N} \right)\right) \Big|\\
        + & \Big|{\dfrac{1}{\# \Lambda_N}} \sum_{x\in \Lambda_{N} }  G(\tfrac{x}{N}) F\left(A[\vf_N, I_{\ve N}(x)],\tfrac{x}{N} \right) - F\left( \left[\Phi(\alpha)+\Phi(\beta) \right] \bar\rho\left(\tfrac{x}{N}\right),\tfrac{x}{N} \right)\Big|\\
        + & \Big|{\dfrac{1}{\# \Lambda_N}} \sum_{x\in \Lambda_{N} }  G(\tfrac{x}{N}) F\left( \left[\Phi(\alpha)+\Phi(\beta) \right] \bar\rho\left(\tfrac{x}{N}\right),\tfrac{x}{N} \right) - \int_0^1 G(u)F\left( \left[ \Phi(\alpha)+ \Phi(\beta)\right]  \bar\rho(u),u \right)du\Big|.
    \end{split}
\end{equation*}
Note that the last sum is a Riemann sum,  { and since $\bar\rho$ is a continuous function, see Lemma \ref{lem:continuity},} the last line in last display  vanishes as $N$ goes to $\infty$. So, it is enough to prove that  the limit as $N\to+\infty$ and then $\ve\to 0$ of the remaining terms vanishes. 
By the triangular inequality, \eqref{eq:centralequation}  and by using the fact $F$ is Lipschitz in the first component we have that 
{\begin{equation*}
    \begin{split}
      &\left\vert {\dfrac{1}{\# \Lambda_N}} \sum_{x\in \Lambda_{N} }  G(\tfrac{x}{N})\left( F(\varphi_{N}(x),\tfrac{x}{N})- F\left(A[\vf_N, I_{\ve N}(x)],\tfrac{x}{N} \right)\right) \right\vert \\
      &\quad \quad + \left\vert {\dfrac{1}{\# \Lambda_N}} \sum_{x\in \Lambda_{N} }  G(\tfrac{x}{N}) \left\{ F\left(A[\vf_N, I_{\ve N}(x)],\tfrac{x}{N} \right) - F\left( \left[\Phi (\alpha)+\Phi(\beta) \right]\bar\rho\left(\tfrac{x}{N}\right),\tfrac{x}{N} \right)\right\}\right\vert\\
        \lesssim &{\dfrac{1}{\# \Lambda_N}} \sum_{x\in \Lambda_{N} } \left\vert G(\tfrac{x}{N}) \right\vert \left\vert E_{\ex} \big[\eta(x) - A[\eta, I_{\ve N}(x)]  \big]\right\vert\\
        &\quad \quad
        + \left\vert{\dfrac{1}{\# \Lambda_N}} \sum_{x\in \Lambda_{N} }  G(\tfrac{x}{N}) \left\{ F\left(A[\vf_N, I_{\ve N}(x)],\tfrac{x}{N} \right) - F\left(\left[ \Phi (\alpha) +\Phi (\beta)\right] \bar\rho\left(\tfrac{x}{N}\right),\tfrac{x}{N} \right) \right\}\right\vert\\
        &= {\rm{(I)}} + {\rm{(II)}},
    \end{split}
\end{equation*}}
where above we have used the notation  $f(t) \lesssim g(t)$ to express the fact there exists a constant  $C$  independent of  $t$  such that $f(t) \leq C g(t)$, for every $t$.

We show separately that ${\rm{(I)}}$ and $\rm{(II)}$ go to zero in order to conclude the proof. For $\rm{(I)}$, by the stationary property of $\mu_{ss}^N$ and Fubini's Theorem, we obtain that
{\begin{equation}
    \label{eq:blablabla}
\left\vert E_{\ex} \left[ \eta(x) - A[\eta, I_{\ve N}(x)]  \right]\right\vert \leq   {\mathbb E}_{\ex}\left[\Bigg\vert  \int_0^1  \eta_t(x) - A[\eta_t, I_{\ve N}(x)]  dt\Bigg\vert   \right].
\end{equation}
By Lemmas 5.3, 5.4 and 5.5 in \cite{BPS} (taking  $\mu_N = \ex$) we obtain that 
$$\lim _{\ve \to 0}\lim _{N\to \infty}   {\mathbb E}_{\ex}\left[\Bigg\vert \int_0^1  \eta_t(x) - A[\eta_t, I_{\ve N}(x)]  dt \Bigg\vert  \right]= 0,$$}
which means that, we can replace the occupation number at site $x$ by its average in a box of length $\#I_{\ve N}(x)$. Hence $\rm{(I)}$ goes to zero. For $\rm{(II)}$, we introduce the notation $\Vert G\Vert_{1,N} ={\dfrac{1}{\# \Lambda_N}} \sum_{x\in \Lambda_N} \vert G (\tfrac x N)\vert$, and use the triangular inequality and the fact $F$ is Lipschitz in the first component, to get the bound 
\begin{equation*}
\begin{split}
   {\rm(II)} &\lesssim\left\vert {\dfrac{1}{\# \Lambda_N}}\sum_{x\in \Lambda_{N} }  G(\tfrac{x}{N}) \left\{ F\left(A[\vf_N, I_{\ve N}(x)],\tfrac{x}{N} \right) - F\left(\tfrac{\Phi(\alpha) +\Phi (\beta)}{\Vert G\Vert_{1,N}} \int_0^1 \vert G(u)\vert {\bar \rho} (u) du ,\tfrac{x}{N} \right)\right\}\right\vert\\
    +&\left\vert {\dfrac{1}{\# \Lambda_N}}\sum_{x\in \Lambda_{N} }  G(\tfrac{x}{N}) \left\{  F\left(\tfrac{\Phi(\alpha)+\Phi(\beta)}{\Vert G\Vert_{1,N}} \int_0^1 \vert G(u)\vert {\bar \rho} (u) du ,\tfrac{x}{N} \right) - F\left( \left[\Phi(\alpha)+\Phi(\beta) \right]\bar\rho\left(\tfrac{x}{N}\right),\tfrac{x}{N} \right)\right\}\right\vert\\
    \lesssim & {\dfrac{1}{\# \Lambda_N}} \sum_{x\in \Lambda_N} \vert G(\tfrac{x}{N}) \vert \left\vert A[\vf_N, I_{\ve N}(x)] - \tfrac{\Phi(\alpha)+\Phi(\beta)}{\Vert G\Vert_{1,N}} \int_0^1 \vert G(u)\vert {\bar \rho} (u) du \right\vert \\ 
    & +\cfrac{\Phi(\alpha)+\Phi(\beta)}{{\# \Lambda_N}} \sum_{x\in \Lambda_N} \vert G(\tfrac{x}{N}) \vert \left\vert \tfrac{1}{\Vert G\Vert_{1,N}} \int_0^1 \vert G(u)\vert {\bar \rho} (u) du -{\bar \rho} \left( \tfrac{x}{N} \right)\right\vert\\
    \lesssim  &\left\vert {\dfrac{1}{\# \Lambda_N}} \sum_{x\in \Lambda_N} \vert G(\tfrac{x}{N}) \vert  A[\vf_N, I_{\ve N}(x)] - \left[\Phi(\alpha)+\Phi(\beta) \right]\int_0^1 \vert G(u)\vert {\bar \rho} (u) du \right\vert\\
    & +\left[ \Phi(\alpha)+\Phi(\beta) \right] \left\vert \int_0^1 \vert G(u)\vert {\bar \rho} (u) du -{\dfrac{1}{\# \Lambda_N}} \sum_{x\in \Lambda_N} \vert G(\tfrac{x}{N}) \vert  {\bar \rho} \left( \tfrac{x}{N} \right)\right\vert.
\end{split}    
\end{equation*}
Recalling that $A[\vf_N, I_{\ve N}(x)] = \dfrac{1}{\# I_{\ve N}(x)}\sum _{y\in I_{\ve N}(x)}\varphi_N (y)$ and $\varphi_N (y) = [\Phi(\alpha)+\Phi(\beta)] E_{\mu_{ss}^N} \left[ \eta(y)\right]$, from Theorem \ref{th:HLM} and Lemma \ref{lem:continuity}, we conclude  that ${\rm(II)}$ goes to zero.

\medskip

It remains to prove that ${\bar \rho} (\cdot)$ takes values in $[\tilde \alpha, \tilde \beta]$. By Proposition \ref{Invariant measure zr} and Proposition \ref{pro_exzr} we have that for any $x \in \Lambda_N$, 
\begin{equation*}
    E_{\ex} \left[ \eta(x) \right] \in [\tilde \alpha, \tilde \beta].
\end{equation*}
By using \eqref{eq:blablabla} and Theorem \ref{th:HLM} we conclude{\footnote{Theorem \ref{th:HLM} is established for continuous functions and we need to apply it to the non continuous function $\iota^{\varepsilon}_u (\cdot) = (2\varepsilon)^{-1} {\mathbb 1}_{\vert \cdot - u \vert \le \varepsilon}$, $u\in [0,1]$. This can be done by a standard approximation argument.}} that for any continuous positive function $G:[0,1] \to \mathbb R$ we have
\begin{equation*}
    \tilde \alpha \int_0^1 G(u) du \le \int_0^1 G(u) {\bar\rho} (u) du \le    \tilde \beta \int_0^1 G(u) du.
\end{equation*}
Then, for any $v\in [0,1]$, we choose a sequence $(G_k)_{k}$ of positive continuous functions defined on $[0,1]$ and converging in the distributional sense to the Dirac mass on $v$. Applying the last inequality to $G_k$ and letting $k$ going to infinity, we conclude, since $\bar \rho$ is continuous that $\tilde \alpha \le {\bar \rho} (v) \le \tilde \beta$.

 \end{proof}

\section{Proof of Theorem \ref{th:ldp}}\label{sec:5}
To prove this theorem we apply Corollary 4.5.27 of \cite{dembo} and therefore we need to check the following facts:
\begin{enumerate}[1.]
    \item For all $G\in C^0 [0,1]$, denoting 
    \begin{equation}
    \label{eq:lambdaNG}
        \Lambda_{N}(G)=\log E_{\nu^{N}_{ss}}\left[ e^{N\langle\pi,G\rangle} \right],
    \end{equation}
    we have that $\lim_{N\to \infty} \dfrac{\Lambda_N (G)}{N}$ exists, it is equal to $\Lambda (G)$ and it is finite.
    \item   The functional $\Lambda$ is  Gateaux differentiable.
\item $\Lambda: C^0 ([0,1]) \to {\mathbb R}$ is lower semi-continuous.
\item The sequence $ \left\{ \pi_N \right\}_{N\ge 1}$ is exponentially tight. \end{enumerate}

\bigskip

Let us prove these four items. The first one is the content of Proposition \ref{prop:MGF}. For the second one, we recall that the partition function $Z$ is analytic on $(0,\infty)$. Then, for all $G,H\in C^0 ([0,1])$ we have that
    $$\lim_{t\to 0} \frac{\Lambda(G + tH) -\Lambda(G)}{t} = \lim_{t\to 0} \int_{0}^{1} \frac{\log\left(\frac{Z(e^{(G + tH)(u)}{\Phi( \bar m}(u)))}{Z(e^{G(u)}\Phi(\bar m (u)))}\right)}{t} du,$$
    and by dominated convergence theorem we get that

    \begin{equation}
        \begin{split}
            \lim_{t\to 0} \frac{\Lambda(G + tH) -\Lambda(G)}{t} &= \int_{0}^{1}\lim_{t\to 0}  \frac{\log\left(\frac{Z(e^{(G + tH)(u)}\Phi(\bar m (u)))}{Z(e^{G(u)}\Phi (\bar m (u)))}\right)}{t} du
             = \int_0^1 R\left(e^{G(u)}\bar \rho(u)\right)H(u) du,
        \end{split}
    \end{equation}
so that the second item is proved. The third item is trivial since $\Lambda$ is a continuous function. Indeed, by using the fact that $\bar \rho(\cdot)$ is a bounded function and $\log Z$ is continuously differentiable, we have that for any $G,H \in C^0 ([0,1])$, 
\begin{equation*}
\begin{split} 
    \left\vert \Lambda (H) -\Lambda (G) \right\vert & \le \int_0^1 \left\vert \log \left( Z \left(e^{H(u)} \bar \rho (u) \right)\right) -\log \left( Z \left(e^{G(u)} \bar  \rho (u) \right)\right)  \right\vert du   \le C \Vert H-G\Vert_{\infty}
\end{split}
\end{equation*}
where $C = \sup_{\varphi \in [0,c]} \left\vert \dfrac{d}{d\varphi} \ \log \left( Z (\varphi) \right) \right\vert$ with $c=\Vert \bar \rho \Vert_{\infty} e^{\sup\left\{\Vert G\Vert_{\infty},\Vert H \Vert_{\infty} \right\}}$. Therefore if $H\to G$ in $C^0 ([0,1])$ then $\Lambda(H) \to \Lambda (G)$. Hence $\Lambda$ is continuous and therefore lower semi-continuous. It remains to prove the forth item. We recall that ${\mc K}_A=\left\{ \mu \in {\mc M} \; ; \; \vert \mu ([0,1]) \vert \le A \right\}$ is a compact subset of $\mc M$ for the weak-$\star$ topology. Hence to prove the exponential tightness of the sequence $\left\{ \pi^N\right\}_{N\ge 1}$ it is sufficient to prove that
\begin{equation}
    \limsup_{A \to \infty} \limsup_{N\to \infty} {\dfrac{1}{\# \Lambda_N}}\log \ { P}_{\nu_{ss}^N} \left[ \dfrac{1}{N}\sum_{x \in \Lambda_N} \xi (x) \ge A\right] =-\infty.
\end{equation}
By Markov's inequality we have that
\begin{equation*}
 { P}_{\nu_{ss}^N} \left[ {\dfrac{1}{\# \Lambda_N}}\sum_{x \in \Lambda_N} \xi (x) \ge A\right] \le e^{-N A + \tfrac{\Lambda_N ({\mathbf 1})}{N}}   
\end{equation*}
where $\Lambda_N$ is defined in \eqref{eq:lambdaNG} and ${\mathbf 1}$ is the constant function on $[0,1]$ equal to $1$. Since $\tfrac{\Lambda_N (1)}{N} \to \Lambda ({\mathbf 1})$ (see Proposition \ref{prop:MGF}) we get the result.

\bigskip

It remains to prove Proposition \ref{prop:MGF}. Recall that $\Lambda_N$ is defined by \eqref{eq:lambdaNG}.

\begin{prop}
\label{prop:MGF}
For any continuous function $G:[0,1]\to \RR$ we have that
$$\lim_{N\to \infty}\left\vert \frac{\Lambda_{N}(G)}{N} -\Lambda(G)\right\vert = 0.$$
\end{prop}

\begin{proof}

We denote for $\lambda\in \mathbb{R}$ and $x\in\Lambda_N$ the  exponential moment of $\xi (x)$: 
$M_{N}(\lambda,x)=E_{\nu^N_{ss}}\left[ e^{\lambda\xi(x)} \right].$
We have that
$$M_{N}(\lambda,x) =\frac{1}{Z(\varphi_N(x))} \sum_{k=0}^{\infty}\frac{\left( e^{\lambda}\varphi_N(x)\right)^k}{g(k)!} = \frac{Z\left(e^{\lambda}\varphi_N(x)\right)}{Z(\varphi_N(x))}.$$
Observe that since $\nu_{ss}^N$ is product we have that 
\begin{eqnarray*}
\frac{\Lambda_{N}(G)}{N} &=&\frac{1}{N}\log E_{\nu_{ss}^N}\left[ e^{N\langle\pi,G\rangle} \right]=
\frac{1}{N}\log E_{\nu_{ss}^N}\left[ e^{\sum_{x\in\Lambda_N}G\Big(\tfrac{x}{N}\Big)\xi_x} \right]
=\dfrac{1}{N}\log E_{\nu _{ss}^N}\left[\prod_{x\in \Lambda_{N}} e^{G\Big(\tfrac x n\Big )\xi_{x}} \right]\\
&=&\dfrac{1}{N}\sum_{x\in \Lambda_{N}}\log M_N \left( G(\tfrac xN), x \right)
=\dfrac{1}{N}\sum_{x\in \Lambda_{N}}\log \dfrac{Z\left(e^{G\Big(\tfrac x n\Big)}\vf_{N}(x) \right)}{Z(\vf_{N}(x))} \
=\dfrac{1}{N}\sum_{x\in \Lambda_{N}} \mathfrak F\left(  \varphi_N (x), G(\tfrac xN)  \right)
\end{eqnarray*}
where $\mathfrak F(\varphi,u) = \log \dfrac{Z\left(e^{G(u)}\vf \right)}{Z(\vf)}$. Observe that $\mathfrak F$ is Lipschitz in the first component because $Z:[0,\infty) \to [1,\infty)$ is analytic. Hence we can apply Theorem \ref{thm:gen_conv} and obtain the result.
\end{proof}

\appendix 
\section{Macroscopic properties of NESS in boundary driven long range exclusion}
\label{app:hydrostaticsexclusion}

In this section, we present a summary of the stationary behaviour of the open boundary exclusion process with long jumps. Recall \eqref{genexcl} and note that we consider the process speeded up in the time scale $\Theta(N)$ as in \eqref{eq:time_scale}. Recall that $\tilde \alpha, \tilde \beta\in(0,1)$. 

 \subsection{Hydrostatic limit}
 
 Let ${\mc M}^+$,  be the space of positive measures on $[0,1]$ with total mass bounded by $1$ and  equipped with the weak $\star$-topology. For any $\eta \in \Omega^N$ the empirical measure  ${ \pi}^N_{ex}:={ \pi}^N_{ex} (\eta) \in {\mc M}^+$ is defined by 
\begin{equation}
{\pi}^N_{ex} (\eta) ={\dfrac{1}{\# \Lambda_N}} \sum_{x\in\Lambda_N} \eta_x   \delta_{\tfrac xN}(du).
\end{equation}
Let $P^N$ be the probability measure on ${\mc M}^+$ obtained as the pushforward of $ \mu_{ss}^N$ by $\pi^N_{ex}$. We denote the action of $\pi_{ex}^N \in {\mc M}^+$ on a continuous function $G:[0,1]\to \RR$ by 
$$\langle \pi^N_{ex} , G \rangle = \int_{[0,1]} G(u) \ \pi^N_{ex}  (du).$$

\begin{thm}[Hydrostatic limit in mean]
\label{th:HLM} 
For any $G \in C^0 ([0,1])$  we have
$$\lim_{N\to\infty}  E_{\mu_{ss}^N}\left [ \langle \pi^N_{ex},G\rangle-\int_0^1\bar\rho(u)G(u)du\right]=0,$$
where  
\begin{itemize}
\item [a)] for $\theta <0$ and {$\gamma\in (0,2)$},  
$
{\bar  \rho} (u) = \dfrac{V_{0}(u)}{V_{1}(u)}.
$

\item  [b)] for $\theta =0$ and $\gamma\in (0,2)-\{1\}$,  $\bar \rho(\cdot)$ is the unique weak  solution of  \eqref{DRex_ZR}, with $\hat \kappa=\kappa$.

\item  [c)] for $\theta \in (0,\gamma-1)$ and $\gamma\in (1,2)$, $\bar \rho(\cdot)$ is the unique  weak solution of \eqref{Dex_ZR}.

 \item  [d)] for  $\theta = \gamma-1$ and $\gamma\in (1,2)$, $\bar \rho(\cdot)$ is the weak  solution of \eqref{Rex_ZR} with {$\hat \kappa=\kappa d$}.

\item  [e)]  for $\theta > 0$ and { $\gamma\in (0, 1]$} or for $\theta >  \gamma- 1$ and  $\gamma\in (1, 2)$, $\bar \rho(\cdot)$ is the weak solution of \eqref{Nex_ZR} with {$\hat M =\tfrac{\tilde \alpha+\tilde \beta}{2}$}, i.e. ${\bar \rho}=\tfrac{\tilde \alpha+\tilde \beta}{2} $.
\end{itemize}


   


    
\end{thm}

\begin{proof}

We start by noting that a simple computation based on the fact that the mass of the system is finite, shows that the sequence $\{P^N\}_{N\geq 2}$ is tight 
and that all limit points are concentrated on measures $\pi(du)$ which are absolutely continuous with respect to Lebesgue measure on $[0, 1]$, i.e. $\pi(du)=\rho(u)du$. Let us also introduce $\{\bar\pi^N\}_{N\geq 2}= \left\{ E_{\mu_{ss}^N} [\pi_{ex}^N] \right\}_{N \ge 2}$ which forms a sequentially compact sequence of ${\mathcal M}^+$ whose limit points $\bar \pi (du)$ are absolutely continuous with respect to the Lebesgue measure on $[0, 1]$, i.e. $\bar\pi(du)=\bar\rho(u)du$. Let $\bar\pi(du)=\bar\rho(u)du$ be a limit point of $\{\bar\pi^N\}_{N\geq 2}$. Without loss of generality we can consider a subsequence for which $\{P^N\}_{N\geq 2}$ is also converging to a limit point denoted by $P^*$ (the corresponding expectation is denoted by ${E} ^*$). To lighten notation, in the sequel, we assume that we are taking the limit according to this subsequence, even if it is not specified.  Observe that $\bar \rho = E^* [\rho]$. Our goal is to show that $\bar\rho$ is unique and given as in Theorem \ref{th:HLM}. 

\medskip
If $\gamma \in (0,2)$, the energy estimates of Section 3.3 of \cite{cpb} show that  for $\theta\geq 0$
$$E^*\left[\Vert \rho\Vert_{\gamma/2}^2\right] <\infty$$
and for $\theta\leq 0$ 
$$ E^*\left[\int_0^1 \left\{\frac{(\tilde \alpha -\rho (u))^2}{u^\gamma} + \frac{(\tilde \beta -\rho (u))^2}{(1-u)^\gamma} \right\} du \right] < \infty.$$
From Jensen's  inequality, we have also 
$$\Vert \bar\rho\Vert_{\gamma/2} <\infty, \quad \int_0^1 \left\{\frac{(\tilde \alpha -\bar\rho (u))^2}{u^\gamma} + \frac{(\tilde \beta -\bar\rho (u))^2}{(1-u)^\gamma} \right\} du < \infty .$$ 

\medskip
It remains now to check that $\bar \rho$ satisfies the other conditions in the notions of stationary weak solutions.

\medskip
Recall \eqref{eq:time_scale}. Note that 
\begin{equation}\label{Dyn_ex}
\begin{aligned}
\Theta(N) \Lex\left\langle{\pi}^N_{ex}, G\right\rangle &=\frac{\Theta(N)}{\# \Lambda_N} \sum_{x \in \Lambda_{N}}\left(\mathbb{L}_{N} G\right)\left(\tfrac{x}{N}\right) \eta (x) \\
&+\frac{\kappa \Theta(N)}{ N^{\theta} \# \Lambda_N} \sum_{x \in \Lambda_{N}} G\left(\tfrac{x}{N}\right)\left\{r_{N}^{-}\left(\tfrac{x}{N}\right)\left(\xalpha-\eta (x)\right)+r_{N}^{+}\left(\tfrac{x}{N}\right)\left(\xbeta-\eta (x)\right)\right\}
\end{aligned}
\end{equation}
where the action of $\mathbb{L}_{N}$ on functions $G$ is defined by
\begin{equation}\label{discrete}
\left(\mathbb{L}_{N} G\right)\left(\tfrac{x}{N}\right)=\sum_{y \in \Lambda_{N}} p(y-x)\left[G\left(\tfrac{y}{N}\right)-G\left(\tfrac{x}{N}\right)\right].
\end{equation}
Taking   the  expectation  with  respect  to $\mu_{ss}^N$
on \eqref{Dyn_ex}, we get, from stationarity, that
\begin{equation}
\label{Dyn_ex_n}
\begin{aligned}
0 &=\frac{\Theta(N)}{\# \Lambda_N} \sum_{x \in \Lambda_{N}}\left(\mathbb{L}_{N} G\right)\left(\tfrac{x}{N}\right) E_{\mu_{ss}^N}[\eta(x)] \\
&+\frac{\kappa \Theta(N)}{ N^{\theta} \# \Lambda_N} \sum_{x \in \Lambda_{N}} G\left(\tfrac{x}{N}\right)\left\{r_{N}^{-}\left(\tfrac{x}{N}\right)\left(\xalpha-E_{\mu_{ss}^N}[\eta(x)]\right)+r_{N}^{+}\left(\tfrac{x}{N}\right)\left(\xbeta-E_{\mu_{ss}^N}[\eta(x)]\right)\right\}.
\end{aligned}
\end{equation}
Recall \eqref{eq:functions_r}. We define the functions $r_{N}^{\pm}: [0,1]\to \RR$ as the linear interpolation of $r_N^- (\tfrac{x}{N})$ and $ r_N^+ (\tfrac{x}{N})$
for all $x \in \Lambda_{N}$ with $r_{N}^{\pm}(0) = r_{N}^{\pm}(\tfrac{1}{N})$ and $r_{N}^{\pm}(1) = r_{N}^{\pm}(\tfrac{N-1}{N})$.
By Lemma 3.3 in  \cite{BJ}, for $0<\gamma <2$ we have that
\begin{equation}
\label{F_convergencer}
\begin{split}
&\lim _{N\to \infty} N^{\gamma} (r_N^-)(u) =r^-(u),\,\,\,\lim _{N\to \infty} N^{\gamma} (r_N^+)(u) =r^+(u)
\end{split} 
\end{equation}
uniformly in  $[a,1-a]$ for $a\in (0,1)$ and  from that {lemma} it also follows  that 
\begin{equation}
\label{F_convergenceLL}
\lim _{N\to \infty}N^{\gamma}({\bb L}_N G)(u)  = (\LL G)(u)
\end{equation}
uniformly in  $[a,1-a]$, for all functions $G$ with compact support included in $[a,1-a]$. 

\bigskip 
Now, we split the analysis by taking into account the value of $\theta$.
\bigskip

\textbf{ Case  $\theta < 0$:} In this regime we take  $G\in C_c^\infty((0,1))$ and  $\Theta(N) = N^{\theta+\gamma}$. {Observe that \eqref{Dyn_ex_n} can be written as 
\begin{equation}
\begin{aligned}
0 =\langle \bar\pi^N,N^{\gamma+\theta}\mathbb{L}_{N} G\rangle -\kappa \langle \bar\pi^N, G N^\gamma(r_N^-+r_N^+)\rangle
+\frac{\kappa}{\# \Lambda_N } \sum_{x \in \Lambda_{N}} G\left(\tfrac{x}{N}\right) N^\gamma\left\{r_{N}^{-}\left(\tfrac{x}{N}\right)\xalpha+r_{N}^{+}\left(\tfrac{x}{N}\right)\xbeta\right\}
\end{aligned}
\end{equation}
}
From the weak convergence together with \eqref{F_convergencer} and \eqref{F_convergenceLL}, last display converges, as $N\to+\infty,$ to 
    $$-\kappa\int_0^1 G(u) \left\{ \bar\rho(u)V_1(u) du - V_0(u)\right\} du =0$$
  {which implies that $\bar\rho(u)=\frac{V_0(u)}{V_1(u)}$ for almost every $u$  (see also Remark (2.14) in \cite{cpb})}.
   
  \medskip
\textbf{  Case  $\theta = 0$:} In this regime we take $G\in C_c^\infty((0,1))$ and  $\Theta (N) = N^{\gamma}$. From weak convergence together with \eqref{F_convergencer} and \eqref{F_convergenceLL} we get that 
\begin{equation*}
    \begin{split}
        0=\int_0^1 (\mathbb{L}G)(u)\bar\rho(u)du -\kappa\int_0^1 G(u)\bar\rho(u)V_1(u) du +\kappa\int_0^1 G(u) V_0(u) du.
    \end{split}
\end{equation*}
Hence we have that $F_{RD} (\bar\rho, G)=0$ for any $G\in C_c^\infty((0,1))$ (with $\hat\kappa=\kappa )$. 

\medskip
\textbf{ Case  $\theta \in (0,\gamma-1)$ and $\gamma \in (1,2)$:}  In this regime we take $G\in C_c^\infty((0,1))$ and $\Theta(N) = N^{\gamma}$. From the weak convergence, the  rightmost term in the first line of  \eqref{Dyn_ex_n} converges, as $N\to+\infty$, to
    $$\int_0^1 (\mathbb{L}G)(u)\bar\rho(u)du.$$
    Moreover, the term on the second line of \eqref{Dyn_ex}  can be bounded from above by a constant times 
   $$ N^{\gamma-\theta-1}\sum_{x\in \Lambda_N} x^{-\gamma}G(\tfrac{x}{N})
     \lesssim N^{-\theta} 
    $$ plus lower order terms in $N$ {which} vanish as $N\to+\infty$. Hence we have that $F_{Dir} (\bar\rho, G)=0$ for any $G\in C_c^\infty((0,1))$.

\medskip
\textbf{  Case  $\theta = \gamma-1$ and $\gamma \in (1,2)$: } In this regime we take  $G\in C^\infty([0,1])$ and $\Theta(N) = N^{\gamma}$. We start by noting that from Lemma 
5.1 of \cite{BPS}, we have that
$$\lim _{N \rightarrow \infty} N^{-1} \sum_{x \in \Lambda_{N}}\left|N^{\gamma}\left(\mathbb{L}_{N} G\right)\left(\tfrac{x}{N}\right)-(\mathbb{L} G)\left(\tfrac{x}{N}\right)\right|=0$$
for functions $G\in C^\infty([0,1])$ and $\gamma\in (0,2).$ Moreover, the first term on the second line of  \eqref{Dyn_ex_n} can be rewritten as 
\begin{equation}\label{eq_new_rr}
    \begin{split}
        &\frac{\kappa N^{\gamma}}{ N^{\gamma-1} \# \Lambda_N}\left(\xalpha-E_{\mu_{ss}^N}\Big[{A}[\eta,I_{\varepsilon N}(0)]\Big]\right) \sum_{x \in \Lambda_{N}} G\left(\tfrac{x}{N}\right) r_{N}^{-}\left(\tfrac{x}{N}\right)\\
        +&\frac{\kappa N^{\gamma}}{ N^{\gamma-1} \# \Lambda_N} \sum_{x \in \Lambda_{N}} G\left(\tfrac{x}{N}\right) r_{N}^{-}\left(\tfrac{x}{N}\right)\left(E_{\mu_{ss}^N}\Big[ {A}[\eta,I_{\varepsilon N}(0)]\Big]-\eta (x)\right)
    \end{split}
\end{equation}
plus analogous terms with respect to the right boundary. 
The expectation of the second term above converges to zero as $N\to+\infty$ from Lemma  5.9 of \cite{BPS} {(see Remark 5.10 there)}. {We note that, to apply Lemma 5.9 of \cite{BPS}, which is written with the time integral and in the $L^1$ sense, in last argument, we need to use the fact that $\mu_{ss}^N$ is a stationary measure, to introduce the time integral in the rightmost term of \eqref{eq_new_rr} and then use the aforementioned lemma.} Now, for  the first  term, we can perform a Taylor expansion on $G$ obtaining the following expression,
\begin{equation*}
\begin{split}
&\frac{\kappa N^{\gamma}G(0)\left(\xalpha- E_{\mu_{ss}^N}\Big[ {A}[\eta,I_{\varepsilon N}(0)]\Big]\right)}{N^{\gamma-1} \ \# \Lambda_N}  \sum_{x \in \Lambda_{N}} r_{N}^{-}\left(\tfrac{x}{N}\right)
+\frac{\kappa  G^{\prime}(0)\left(\xalpha-E_{\mu_{ss}^N}\Big[ {A}[\eta,I_{\varepsilon N}(0)]\Big]\right)}{\# \Lambda_N} \sum_{x \in \Lambda_{N}} x r_{N}^{-}\left(\tfrac{x}{N}\right)
\end{split}
\end{equation*}
plus lower order terms in $N$. Observe that 
\begin{equation*}
  \begin{split}
\dfrac{\kappa  \left\vert G^{\prime}(0)\left(\xalpha-E_{\mu_{ss}^N}\Big[ {A}[\eta,I_{\varepsilon N}(0)]\Big]\right) \right\vert}{\# \Lambda_N} \sum_{x \in \Lambda_{N}} x r_{N}^{-}\left(\tfrac{x}{N}\right)
\lesssim &\frac{1}{\# \Lambda_N} \sum_{x\in \Lambda_N} x^{1-\gamma} = \mathcal{O}(N^{1-\gamma}),
  \end{split}
\end{equation*}
which goes to zero as $N$ goes to infinity. {Finally, the remaining term can be treated by using the weak convergence, the fact that the limiting measure $\bar\pi(du)$ is absolutely continuous with respect to the Lebesgue measure with density $\bar\rho$ and also that $$\lim_{N\to+\infty}\sum_{x\in\Lambda_N}r_N^\pm(\tfrac xN)=d.$$} Hence we get that $F_{Rob} (\bar\rho, G)=0$ (with $\hat \kappa =\kappa d$) for any $G\in C^\infty([0,1])$.

\medskip
\textbf{  Case  $\theta >\gamma-1$ {and $\gamma\in(1,2)$}: } 
The proof in this regime is completely analogous to the previous case. Nevertheless,  we could also obtain the result as a consequence of the hydrostatic limit, with a convergence in probability, by following the same strategy described in \cite{Tsunoda}. Since we do not need this stronger convergence to attain  our results, we did not pursue this issue here. 

\medskip
\textbf{  Case  $\theta > 0$ and  $\gamma \in (0,1)$: } 
In this case the analysis of the first term of {\eqref{Dyn_ex_n} is given by an approximation argument of the operator $\mathbb{L}$ given in Lemma 5.1 of \cite{BPS}. }
Since $G(x/N)$ and $ \eta (x) $ are bounded we know that the second term of \eqref{Dyn_ex_n} is bounded by a constant times 
$$N^{\gamma-\theta-1} \sum_{x \in \Lambda_{N}} \left\{ r_{N}^{-}\left(\tfrac{x}{N}\right) +r_{N}^{+}\left(\tfrac{x}{N}\right)\right\}.$$
Since $\gamma\in (0,1)$, we lose the convergence of the partial sum above. However,   it is not difficult to see that 
\begin{equation}
    \begin{split}
       & N^{\gamma-\theta-1} \sum_{x \in \Lambda_{N}} \left\{ r_{N}^{-}\left(\tfrac{x}{N}\right) +r_{N}^{+}\left(\tfrac{x}{N}\right)\right\}
      \lesssim N^{-\theta},
    \end{split}
\end{equation}
which vanishes as $N$ goes to $\infty$. Hence $F_{Neu} (\bar\rho, G)=0$ for any $G\in C^\infty([0,1])$.

\medskip
In the two last cases, we have also to show that $\int_0^1 {\bar\rho} (u) du =\tfrac{\tilde\alpha + \tilde\beta}{2}$. In fact, since in these cases, $F_{Neu} (\bar\rho, G)=0$ for any $G\in C^\infty([0,1])$, we can conclude that $\bar\rho$ is equal to a constant $M$ by using  the same argument as in Lemma \ref{lem:214} when showing that $\Vert {\bar \rho}\Vert_{\gamma/2}=0$.  



Recall the definition of ${\bar\pi}^N$ given at the beginning of the proof. We consider a sequence $(\chi^\varepsilon)_{0<\varepsilon< 1/4}$ of smooth functions with values in $[0,1]$, symmetric with respect to $1/2$, equal to $0$ on $[0,\varepsilon/2]$ and to $1$ on $[\varepsilon,1/2]$. Observe that $(\chi^\varepsilon)_\varepsilon$ converges in $L^1$ as $\varepsilon$ goes to $0$ to the constant function equal to $1$. Taking $G=\chi^\varepsilon$ in \eqref{Dyn_ex_n} we get 
\begin{equation*}
\begin{split}
     &\frac{1}{N}\sum_{x \in \Lambda_{N}} \chi^\varepsilon \left( \tfrac xN\right)\left\{\tilde\alpha N^\gamma r_{N}^{-}\left(\tfrac{x}{N}\right)+ \tilde\beta N^\gamma r_{N}^{+}\left(\tfrac{x}{N}\right)\right\}\\
     & = \int_0^1 \chi^\varepsilon (u)  (N^\gamma r_N^- (u) +N^\gamma r_N^+ (u)) {\bar\pi}^N (du) + N^{-\theta} \frac1N \sum_{x \in \Lambda_N} (N^\gamma \mathbb L_N \chi^\varepsilon) \left( \tfrac xN \right) E_{\mu_{ss}^N} \left[ \eta (x)\right].
\end{split}
\end{equation*}  
By \eqref{F_convergenceLL} and \eqref{F_convergencer}, we can then replace in the previous expression $N^\gamma \mathbb L_N$ by $\mathbb L$ and $N^\gamma r_N^\pm$ by $r^\pm$. Recall that ${\bar\pi}^N (du)$ converges weakly to $\bar\rho (u) du=M du$. Since $\theta>0$ we get 
\begin{equation}
     \int_0^1 \chi^\varepsilon (u) V_1 (u) du = M \int_0^1 \chi^\varepsilon (u) V_0 (u) du.
\end{equation}
In the case $\gamma \in (0,1)$ the functions $V_0$ and $V_1$ are integrable and we get by sending $\varepsilon$ to $0$ that
$$M=\frac{\int_0^1 V_1 (u) du}{\int_0^1 V_0 (u) du}=\frac{\tilde\alpha + \tilde\beta}{2}.$$
In the case $\gamma \in (1,2)$ the integrals are diverging but 
\begin{equation*}
 \int_0^1 \chi^\varepsilon (u) V_1 (u) du \sim \cfrac{c_\gamma}{\gamma (\gamma-1)}\varepsilon^{1-\gamma}\left[\tilde \alpha +\tilde\beta\right], \quad \int_0^1 \chi^\varepsilon (u) V_0 (u) du \sim 2 \cfrac{c_\gamma}{\gamma (\gamma-1)}\varepsilon^{1-\gamma}\ 
\end{equation*}
therefore we get again that 
$$M=\frac{\tilde\alpha + \tilde\beta}{2}.$$

\medskip 

\textbf{Conclusion:} All limit points of the sequence $\{\bar\pi^N (du) \}_{N\ge 2}$ are in the form $\bar\rho (u) du$ where $\bar\rho$ is a weak solution of the hydrostatic equation. By uniqueness of weak solutions for these equations, $\bar \rho$ is unique and therefore the sequence is converging, without extracting a subsequence,  to this unique weak solution.
\end{proof}

\begin{lem}
\label{lem:continuity}
The profiles $\bar \rho$ in Theorem \ref{th:HLM} are continuous in $(0,1)$. 
\end{lem}

\begin{proof}
In the case $\theta<0$ the claim follows easily since the profile is explicit. We consider now $\theta\ge 0$. If $\gamma \in (1,2)$, by definition of a weak solution, we know that $\bar \rho$ is bounded and belongs to $\mathcal H^{\gamma/2}$ and from Theorem 8.2 of \cite{DNPV}, we conclude that $\bar \rho $ is $\tfrac{\gamma-1}{2}$-Hölder in $[0,1]$, therefore continuous in $(0,1)$. If $\gamma \in (0,1)$ and $\theta >0$, the profile is constant and therefore continuous. The only missing case  is $\gamma \in (0,1)$ and $\theta=0$. We have hence to prove that the stationary solution $\rho$ of the regional fractional reaction-diffusion equation \eqref{DRex_ZR} is continuous in $(0,1)$ when $\gamma\in (0,1)$. It is known that if $\gamma \in (0,1]$, the condition $f \in {\mathcal H}^{\gamma/2}$ does not guarantee, contrarily to the case $\gamma \in (1,2)$, that $f$ is continuous. Therefore the continuity property of $\rho$ can only result from the fact that $\rho$ satisfies the weak formulation of \eqref{DRex_ZR}.
This property is a consequence of potential theory for (fractional) Shr\"odinger theory developed in \cite{BB}, more exactly of Proposition 6.1 of that article that we restate in our particular context. Before doing so,  we introduce a  few notations. 

The fractional Laplacian $\vert \Delta\vert^{\gamma/2}$ on $\mathbb R$ is the operator acting on functions $f:\mathbb R \rightarrow \RR$ such that
\begin{equation}
\label{eq:integrability}
\int_{\mathbb R}  \cfrac{|f(u)|}{(1 +|u|)^{1+\gamma}} du < \infty
\end{equation}
as
\begin{equation}
\label{definition}
-(\vert \Delta\vert^{\gamma/2} f)(u) = c_\gamma  \lim_{\epsilon \to 0} \int_{\mathbb R} {\bb 1}_{|u-v| \ge \epsilon} \cfrac{f(v) -f(u)}{|u-v|^{1+\gamma}} dv,
\end{equation}
for any $u \in \mathbb R$ if  the limit exists (which is for example the case for smooth compactly supported functions). It can be extended into the weak fractional Laplacian (that we denote abusively by the same notation) by duality: For any $f$ satisfying \eqref{eq:integrability}, $\vert \Delta\vert^{\gamma/2} f$ is the distribution (or generalized function) on $\mathbb R$ satisfying the identity  $\langle \vert \Delta\vert^{\gamma/2} f, h \rangle = \langle f, \vert \Delta\vert^{\gamma/2} h \rangle$,  for any compactly supported function $h$ on $\mathbb R$. 

Property 6.1 of \cite{BB} claims that if $f:\mathbb R \to \mathbb R$ is a solution, in the distributional sense{\footnote{It means that for any $h \in C_c^\infty ((0,1))$, $ \int_0^1 f(u) \vert \Delta\vert^{\gamma/2} h (u) \, du \, +\, \int_0^1 f(u) q(u) \, du =0 $. }}, on $(0,1)$ of the equation
\begin{equation}
\label{eq:skato}
   \vert \Delta\vert^{\gamma/2} f + q f =0
\end{equation}
then $f$ is continuous on $(0,1)$ as soon as the function $q:(0,1) \to \mathbb R$ belongs to the (local) Kato class of exponent $\gamma$, i.e.
\begin{equation*}
\lim_{r\to 0}\,  \sup_{u\in \mathbb R} \int_{u-r}^{u+r} \cfrac{ ({\mathbb 1}_{[\varepsilon, 1-\varepsilon]} q) (u)  }{|u-v|^{1-\gamma}} \ dv =0  
\end{equation*}
for any $\varepsilon \in (0,1/2)$. It is not difficult to see that if $\rho$ is the weak solution of \eqref{DRex_ZR} and is extended by $0$ outside of $(0,1)$, then $\rho$ satisfies, in the distributional sense, \eqref{eq:skato} on $(0,1)$ for $q$ given as a linear combination of $r^-$ and $r^+$ (defined by \eqref{def:rpm}). Hence to conclude the proof of the lemma, it is sufficient to prove that $r^\pm$ belong to the (local) Kato class of exponent $\gamma$. This  exercise is left to the interested reader.
\end{proof}
\subsection{Fractional Fick's Law of the boundary driven exclusion}
 
By adapting the strategy of \cite{BJ} we can obtain the  ``fractional Fick's Law'' which is given in the next theorem. The expression of the current is given by 
\begin{equation}
\label{eq:currentexclusion0}
\begin{split}
W_x^{ex}(\eta)=  &\sum_{\substack {1 \le y\le x-1 \\
 x-1<z\le N-1}}  p(z-y) (\eta(y)  -\eta(z) ) \\
+&{\frac{\kappa}{ N^{\theta}}}\Bigg[\sum_{x \le z \le N-1} r_N^-(\tfrac zN) ( \tilde\alpha -\eta(z)) -\sum_{1 \le y\le x-1}r_N^+(\tfrac yN)   (\tilde\beta-\eta(y))\Bigg].
\end{split}
\end{equation}
\begin{thm}(Fractional Fick's law)
\label{thm:Fick}

Let $\bar \rho(\cdot)$ be the hydrostatic profile of the boundary driven exclusion given in Theorem \ref{th:HLM}. 
For $u\in (0,1)$ the following fractional Fick's law holds, {apart from the case $\theta=0$ and $\gamma=1$} :
\begin{equation}\label{lim_fick}\lim_{N\to\infty}\frac{1}{B_N(\theta)} {E}_{
\mu_{ss}^N}[ W_{[uN]}^{ex}] = \int_0^1 h_\theta(u)\bar\rho(u)du + C(\xalpha,\xbeta,\theta),\end{equation}
where
\begin{eqnarray}
\label{EPB}
B_N(\theta): =    N^{1-\gamma}\bb{1}_{\theta \geq 0}+
      {N^{1-\theta-\gamma}}\bb {1}_{\theta < 0},
\end{eqnarray}
 the function $h_\theta:(0,1)\to\mathbb{R}$ is given by
 {\begin{equation} \label{function_h}
h_\theta(u)=\begin{cases}
c_\gamma\left(\dfrac{\kappa }{\gamma}\textbf{1}_{\theta\leq 0} {+}\dfrac{1}{1-\gamma}\textbf{1}_{\theta\geq 0}\right)[(1-u)^{1-\gamma}-u^{1-\gamma}], \quad \textrm{if} \quad \gamma\neq 1 ,\\c_\gamma 
\textbf{1}_{\theta\geq 0}[\log(1-u)-\log(u)], \quad \textrm{if} \quad \gamma= 1 
\end{cases}\end{equation}}

and 
\begin{equation}\label{constant_c}  C(\xalpha,\xbeta, \theta) =  
      \dfrac{c_\gamma \kappa(\xalpha -\xbeta)}{\gamma(2-\gamma)}\textbf{1}_{ \theta \leq  0, }
      \end{equation}
      
This implies that
\begin{itemize}
    \item [a)] for $\theta < 0$,
 \begin{equation}
 \label{eq:fl67}
 \begin{split}
 \lim_{N \to \infty}\frac{1}{N^{1-\theta-\gamma}}\EE_{ \ex }[{W_{[uN]}^{ex}}] &= \kappa\int_{u}^{1}(\xalpha -\bar\rho(v))r^{-}(v)dv- \kappa\int_{0}^{u}(\xbeta -\bar\rho(v))r^{+}(v)dv\\
&=\kappa c_\gamma \gamma^{-1} \int_0^1\dfrac{\xalpha-\xbeta}{v^\gamma +(1-v)^\gamma}dv;
 \end{split}
 \end{equation}

 \item [b)] for $\theta = 0$,
\begin{equation}
 \label{eq:fl68}
 \begin{split}
 \lim_{N \to \infty} \frac{1}{N^{1-\gamma}} \EE_{\ex }[{W_{[uN]}^{ex}}] =c_{\gamma} \int_{0}^u \; \int_{u}^{1} \, \cfrac{{ \bar\rho} (v) -{ \bar\rho} (w)}{(w-v)^{1+\gamma}}\, dw dv &+ \kappa\int_{u}^{1}(\xalpha -\bar\rho(v))r^{-}(v)dv\\
&- \kappa\int_{0}^{u}(\xbeta -\bar\rho(v))r^{+}(v)dv;
 \end{split}
 \end{equation}
    \item [c)] for $\theta >0$,
\begin{equation}
 \label{eq:fl69}
 \begin{split}
 \lim_{N \to \infty} \frac{1}{N^{1-\gamma}}\EE_{\ex }[{W_{[uN]}^{ex}}] &=c_{\gamma} \int_{0}^u \; \int_{u}^{1} \, \cfrac{{ \bar\rho} (v) -{ \bar\rho} (w)}{(w-v)^{1+\gamma}}\, dw dv.
 \end{split}
 \end{equation}
\end{itemize}
       \end{thm}

\begin{proof} 
Since the measure $\ex$ is stationary  and $\mathcal L_N \eta_ x = W_x^{ex}(\eta)-W_{x+1}^{ex}(\eta)$ for all $x\in \Lambda_N$ and for all $\eta$, it follows that ${E}_{\ex}[ W_x^{ex}] = {E}_{\ex}[ W_1^{ex}],$ for all $x\in \Lambda_N$. Then we can write
{\begin{equation}\label{eq:Curr}
    \begin{split}
    {E}_{\ex}[ W_1^{ex}] &= \frac{\kappa N^{-\theta}}{\# \Lambda_N}\sum_{z\in\Lambda_N} z \left[ \xalpha - {E}_{\ex}[\eta(z)] \right]\sum_{y\leq 0}p(z-y) \\
    &+ \frac{\kappa N^{-\theta}}{\# \Lambda_N}\sum_{y\in\Lambda_N} (N-1-y) \left[{E}_{\ex}[\eta(y)]-\xbeta \right]\sum_{z\geq N}p(z-y) \\
    &+\frac{1}{\# \Lambda_N}\sum_{z\in\Lambda_N}\sum_{y=1}^{z-1}p(z-y) (z-y)\left[{E}_{\ex}[\eta(y)]-{E}_{\ex}[\eta(z)]\right].
    \end{split}
\end{equation}
}

We define the linear interpolation functions $ \tilde{r}_{N}^{\pm} :[0,1]\to\R$, such that for all $z\in \Lambda_N$ we have that
\begin{equation}\label{tilder}
\tilde{r}_{N}^{-}\left(\tfrac{z}{N}\right)=\sum_{y \geq z} y p(y) \quad \textrm{and} \quad  \tilde{r}_{N}^{+}\left(\tfrac{z}{N}\right)=-\sum_{y \leq z-N} y p(y).\end{equation}

Using \eqref{eq:Curr} it is not difficult to see that
\begin{equation}\label{eq:fk1}
\frac{1}{B_N(\theta)} {E}_{\ex}[ W_1^{ex}] =\frac{1}{\# \Lambda_N} \sum_{x\in\Lambda_N}h_\theta^N\left(\tfrac{x}{N}\right)E_{\mu_{ss}^N}[\eta(x)] + \dfrac{ C_\theta^N}{B_N(\theta)},
\end{equation}
where 
$$ h_\theta^N \left(\tfrac{x}{N}\right)= \frac{\kappa N^{1-\theta}}{B_N(\theta)}\left[ -\tfrac{x}{N}r^-_N\left(\tfrac{x}{N}\right) +\left(\tfrac{N-1-x}{N}\right)r^+_N\left(\tfrac{x}{N}\right)\right] +\frac{\tilde r^-_N\left(\tfrac{x}{N}\right) - \tilde r^+_N\left(\tfrac{x}{N}\right)}{B_N(\theta)},$$
and 
$$C_{\theta}^N:={\frac{\kappa}{ N^{\theta}\ \# \Lambda_N}}\left[  \xalpha \sum_{z\in\Lambda_N}\sum_{y\leq 0}z p(z-y) -\xbeta\sum_{y\in\Lambda_N}\sum_{z\geq N}(N-1-y)p(z-y)\right].$$
From \eqref{tilder},  we get the following convergence
\begin{equation}\label{eq:convpwh}
 \lim_{N\to\infty} \left\vert h_\theta^N\left( \tfrac{[uN]}{N}\right) - h_\theta(u) \right\vert = 0 
\end{equation}
which holds uniformly for $u \in (a,1-a)$ (with $0<a<1$ fixed), and the function $h_\theta:(0,1)\to\mathbb{R}$ is given in \eqref{function_h}.
We note that $h_\theta$ is singular at $u=0$ and $u=1$ if $\gamma\in (1,2)$, but it is integrable in $[0,1]$ for $\gamma \in (0,2)$.   Moreover, it is easy to see that, for $u\in (0,1)$ {and for $\gamma\neq 1$} we have that
\begin{equation}\label{cota_h}
    \vert h_\theta(u)\vert\lesssim u^{1-\gamma}+(1-u)^{1-\gamma}. 
\end{equation}
Regarding the last term of \eqref{eq:fk1}, a simple computation shows that 
$$\displaystyle\lim_{N\to\infty}\frac{1}{B_N(\theta)}C_\theta^N=C(\xalpha,\xbeta,\theta),$$
 where $ C(\xalpha,\xbeta,\theta)$ was defined in \eqref{constant_c}.
Now, note that from \eqref{eq:convpwh} and the fact that $| \eta(x)|\leq 1 $,  we get that
\begin{equation}\label{eq:fk2}
\lim_{N\to \infty}\left\vert \frac{1}{\# \Lambda_N} \sum_{x\in N I(a)}\Big( h_\theta^N\left(\tfrac{x}{N}\right)-h_\theta(\tfrac xN)\Big)E_{\mu_{ss}^N}[\eta(x)] \right\vert  = 0,
\end{equation}
where $I(a) = [a,1-a]$ and $NI(a)= [Na,N(1-a)]\cap \mathbb{N}.$
From \eqref{cota_h} and the fact that $| \eta(x)|\leq 1 $ we get that
$$\left|\frac{1}{\# \Lambda_N} \sum_{x\in (N I(a))^{c}} h_\theta^{N}\left(\tfrac{x}{N}\right) {E}_{\ex} [ \eta(x) ] \right| \lesssim \left[a^{2-\gamma}+(1-a)^{2-\gamma}\right].$$
Moreover, from Theorem \ref{th:HLM} and, for   $h_\theta^a$ a continuous extension of the function $h_\theta$ restricted to ${I(a)}$, we get that
\begin{equation}\label{eq:fk2}
\lim_{N\to \infty}\left\vert \frac{1}{\# \Lambda_N} \sum_{x\in\Lambda_N}h_\theta^a(\tfrac xN)E_{\mu_{ss}^N}[\eta(x)] -\int_0^1 h_\theta^a(u)\bar\rho(u)du\right\vert  = 0.
\end{equation}
Now, from \eqref{eq:fk1} and \eqref{eq:fk2} sending first $N\to \infty$ and then  $a\to 0$ we obtain  \eqref{lim_fick}. 

\medskip
The other expressions of the limiting current given at the end of the theorem are achieved by using properties of the integrals and the fact that the limit does not depend on the variable $u$. To check them properly, let us  consider for instance $\theta>0$ and $\gamma\neq 1$. Since the limit does not depend on $u$, we have that
$$ c_{\gamma} \int_{0}^u \; \int_{u}^{1} \, \cfrac{{ \bar\rho} (v) -{ \bar\rho} (w)}{(w-v)^{1+\gamma}}\, dw dv=c_\gamma \int_0^1\int_0^u\int_u^1 \frac{\bar\rho(v)-\bar\rho(w)}{(w-v)^{\gamma+1}}dwdvdu.$$
Using Fubini's theorem twice, last display equals to
$$c_\gamma \int_0^1\int_v^1\int_v^w \frac{\bar\rho(v)-\bar\rho(w)}{(w-v)^{\gamma+1}}dudwdv=c_\gamma \int_0^1\int_v^1 \frac{\bar\rho(v)-\bar\rho(w)}{(w-v)^{\gamma}}dwdv.$$
Finally, a simple computation, based again on Fubini's theorem, shows that last display is equal to 
$$\int_0^1 h_\theta(v)\bar\rho(v)dv .$$
This ends the case $\theta>0$.  The cases $\theta \leq 0$ can be obtained by performing similar computations to the ones above, plus the fact that 
\begin{equation*}
\kappa\int_{0}^1 \int_{u}^{1}\xalpha r^{-}(v)dv\, du- \kappa\int_{0}^{1}\int_0^u\xbeta r^{+}(v)dv\, du=\frac{c_\gamma \kappa(\xalpha-\xbeta)}{\gamma(2-\gamma)}.
 \end{equation*}
Finally, we note that the second equality in item a) is obtained by algebraic manipulations using the fact  that  $\bar\rho(u) = \dfrac{ V_0(u)}{V_1(u)}$. 
\end{proof}

\section*{Acknowledgement}
{The work of C.B. has been supported by the projects LSD ANR-15-CE40-0020-01 of the French National Research Agency (ANR). B.J.O. thanks  Universidad Nacional de Costa Rica  for sponsoring the participation in  this article through the project 0497-18.
 P.G.  and S.S. thank  FCT/Portugal  for financial support through CAMGSD, IST-ID,
projects UIDB/04459/2020 and UIDP/04459/2020.   This project has received funding from the European Research Council (ERC) under  the European Union's Horizon 2020 research and innovative programme (grant agreement   n. 715734).}

\end{document}